\documentclass[a4,12pt]{article}%

\usepackage{graphicx}
\usepackage{graphics}
\newcounter{figurecounter}
\usepackage{amsmath}
\usepackage{amsfonts}
\usepackage{amssymb}
\usepackage{enumerate}
\newtheorem{theorem}{Theorem}[section]

\newtheorem{corollary}[theorem]{Corollary}

\newtheorem{definition}[theorem]{Definition}

\newtheorem{example}[theorem]{Example}

\newtheorem{lemma}[theorem]{Lemma}

\newtheorem{remark}[theorem]{Remark}

\newenvironment{proof}[1][Proof]{\textbf{#1.} }{\ \rule{0.5em}{0.5em}}
\newcommand{\E}{{\rm \bf E}}

\newcommand{\prob}{{\rm \bf P}}

\newcommand{\supp}{{\rm supp\ }}

\newcommand{\argmax}{{\rm argmax}}

\newcommand{\dN}{{\mathbb N}}

\newcommand{\dR}{{\mathbb R}}

\newcommand{\Sigmastat}{{\Sigma_{\hbox{\small{stat}}}}}
\newcommand{\Sigmaistat}{{\Sigma^i_{\hbox{\small{stat}}}}}
\newcommand{\Sigmamistat}{{\Sigma^{-i}_{\hbox{\small{stat}}}}}
\newcommand{\Sigmaistatft}{{\Sigma^i_{\hbox{\tiny{stat}}}}}
\newcommand{\Sigmamistatft}{{\Sigma^{-i}_{\hbox{\tiny{stat}}}}}

\newcommand{\calD}{{\cal D}}

\newcommand{\calF}{{\cal F}}

\newcommand{\calH}{{\cal H}}
\newcommand{\calI}{{\cal I}}

\newcommand{\ep}{\varepsilon}

\newcommand{\numbercellongg}[2]
{
\begin{picture}(60,20)(0,0)
\put(0,0){\framebox(60,20)} \put(30,10){\makebox(0,0){#1}}
\put(52,12){\makebox(0,0){#2}}
\end{picture}
}

\newcommand{\numbercellonga}[2]
{
\begin{picture}(100,20)(0,0)
\put(0,0){\framebox(100,20)}
\put(10,10){\makebox(0,0){#1}}
\put(60,11){\makebox(0,0){#2}}
\end{picture}
}

\newcommand{\numbercellongb}[2]
{
\begin{picture}(40,20)(0,0)
\put(0,0){\framebox(40,20)}
\put(10,10){\makebox(0,0){#1}}
\put(30,11){\makebox(0,0){#2}}
\end{picture}
}

\begin{document}

\title{The Modified Stochastic Games%
\thanks{The author thanks Eitan Altman for helping in identifying relevant references,
Omri Solan for helpful discussions,
and acknowledges the support of the Israel Science Foundation, Grant \#323/13.}}

\author{Eilon Solan%
\thanks{The School of Mathematical Sciences, Tel Aviv
University, Tel Aviv 6997800, Israel. e-mail: eilons@post.tau.ac.il}}

\maketitle

\begin{abstract}
We present a new tool for the study of multiplayer stochastic games,
namely the modified game,
which is a normal-form game that depends on the discount factor,
the initial state,
and for every player a partition of the set of states
and a vector that assigns a real number to each element of the partition.
We study properties of the modified game, like its equilibria, min-max value, and max-min value.
We then show how this tool can be used to prove the existence of a uniform equilibrium in a certain class of multiplayer stochastic games.
\end{abstract}



\section{Introduction}

Stochastic games is a model for studying dynamic models in which
the state variable changes in response to the players' actions.
Shapley (1953) presented the model of two-player zero-sum discounted stochastic games with finitely many states and actions
and proved the existence of the value and of stationary optimal strategies for the two players.
The main approach for studying discounted stochastic games is using their recursive structure.
This approach was successfully utilized by Fink (1964) and Takahashi (1964)
to prove the existence of a discounted equilibrium in stationary strategy in multiplayer stochastic games,
by H\"orner, Sugaya, Takahashi, and Vieille (2011) for characterizing the limit set of discounted equilibrium payoffs in a certain class of multiplayer stochastic games,
and by, e.g., Mertens and Parthasarathy (1987) to study equilibrium in multiplayer stochastic games with general state and action spaces.

Mertens and Neyman (1981) suggested a robust equilibrium concept for stochastic games, namely uniform $\ep$-equilibrium.
A strategy profile is a uniform $\ep$-equilibrium if it is an $\ep$-equilibrium in the discounted game, provided the players are sufficiently patient,
and in the finite-horizon game, provided the game is sufficiently long.
The study of uniform $\ep$-equilibrium and its variants, namely various types of uniform correlated $\ep$-equilibrium,
turned out to be difficult, and various techniques were used in the literature by different authors.
These include the vanishing discount factor approach (Vrieze and Thuijsman, 1989),
graph theoretic tools (Vieille, 2000a,b),
the introduction of a modified game (Solan, 1999, 2000; Solan and Vohra, 2002),
a dynamical system approach (Solan and Vieille, 2001),
and a topological approach (Simon, 2007, 2012).

In the present paper we extend to general stochastic games the approach taken by Solan (1999, 2000) and Solan and Vohra (2002).
These papers study multiplayer absorbing games,%
\footnote{A state in a stochastic games is \emph{absorbing} if the play cannot leave it once this state was reached.
A stochastic game is \emph{absorbing} if all states but one are absorbing.}
and define a \emph{modified} game, in which the stage payoff in the nonabsorbing state is different than the stage payoff in the original game:
the new stage payoff of each player is the minimum between his \emph{expected} stage payoff (given the mixed actions chosen by the players)
and the uniform min-max value of the player in the nonabsorbing state.
It is shown that this game admits a discounted equilibrium in stationary strategies
and that the limit of the discounted equilibrium payoffs is at least the uniform min-max value of the original game.
Finally, like in the vanishing discount factor approach, the sequence of discounted stationary equilibria
is used to construct a uniform $\ep$-equilibrium in three-player absorbing games and in team games,
and a uniform normal-form correlated $\ep$-equilibrium in multiplayer absorbing games.

The fact that in absorbing games there is a single nonabsorbing state simplified the definition and the study of the modified game,
as the expected stage payoff and the uniform min-max value of the players are those given in the nonabsorbing state.
In general stochastic games, where the state changes from stage to stage,
it is not clear how to define the payoff function in the modified game so as to keep the useful properties
that the definition of the modified game for absorbing games has.

Below we provide a definition for the modified game that retains the features that were valuable in earlier studies.
The modified game is a normal-form game that depends on the initial state, a discount factor,
and for each player $i$ a partition $\calD^i$ of the set of states and a collection of cut-offs $(c^i(D))$,
one for every element $D$ of the partition $\calD^i$.
In this game, each player chooses a strategy in the stochastic game,
and his payoff is the sum over all elements $D$ of the partition,
of the minimum between his expected discounted payoff restricted to stages spent in $D$
and his cut-off $c^i(D)$.
Thus, the expected discounted payoff of a player $i$ in an element $D$ of his partition cannot be higher than the cut-off $c^i(D)$.

We then study properties of the modified game.
We show that it admits an equilibrium in stationary strategies,
and we compare its min-max value and max-min value, both in general strategies and in stationary strategies,
to the min-max value and max-min value of the original game.
In particular we show that if the partitions $(\calD^i)$ satisfy a certain property and the cut-offs $(c^i(D))$ are not too low,
then the limit of a sequence of equilibrium payoffs in the modified game as the discount factor goes to 1 (patience)
is at least the uniform min-max value of the initial state in the original stochastic game.
We finally provide an application of this tool to the study of uniform $\ep$-equilibrium in a certain class of multiplayer stochastic games.

In addition to providing a new tool to study multiplayer stochastic games,
the paper shows that to study stochastic games it may be useful to define auxiliary games and study their properties.
The modified game that we present is just one possible auxiliary game.

The paper is organized as follows.
In Section~\ref{section:model} we describe the model of stochastic games.
In Section~\ref{section:definition} we define the modified game and summarize the results that are proven in the rest of the paper.
In Section~\ref{section:equilibrium} we study equilibria of the modified game.
In Section~\ref{section:max:min} we present the notion of uniform max-min value in stochastic games.
The max-min value and the min-max value of the modified game,
and their comparison to the uniform max-min value and uniform min-max value in the original game, are studied in Sections~\ref{section:maxmin}
and~\ref{section:minmax} respectively.
The min-max and max-min values in stationary strategies of the modified game are studied in Sections~\ref{section:minmax:stationary}
and~\ref{section:maxmin:stationary} respectively.
The application to the uniform equilibrium appears in Section~\ref{section:application}.

\section{Stochastic Games: the Model}
\label{section:model}

\subsection{The Game Play}

A multiplayer \emph{stochastic game} is a vector $\Gamma = (I,S,(A^i)_{i \in I}, (u^i)_{i \in I}, q)$ where
\begin{itemize}
\item   $I = \{1,2,\ldots,|I|\}$ is a finite set of players.
\item   $S$ is a finite set of states.
\item   $A^i(s)$ is a finite set of actions available to player $i$ in state $s$.
Denote by $A(s) := \times_{i \in I} A^i(s)$ the set of all action profiles available at state $s$.
Denote by $\Lambda := \{(s,a) \colon s \in S, a \in A(s)\}$ the set consisting of pairs of state and action profile available at that state.
\item   $u^i : \Lambda \to \dR$  is player $i$'s payoff function.
We assume w.l.o.g.~that the payoffs are bounded between -1 and 1.
\item   $q : \Lambda \to \Delta(S)$ is a transition function, where $\Delta(X)$ is the set of probability distributions over $X$,
for every nonempty finite set $X$.
\end{itemize}

The game is played as follows.
The initial state $s_0 \in S$ is given.
At each stage $n \in \dN \cup \{0\}$, the current state $s_{n}$ is announced to the players. Each
player $i \in I$ chooses an action $a^{i}_{n}\in A^{i}(s_n)$; the action profile
$a_{n}=(a^{i}_{n})_{i\in I}$ is publicly announced, $s^{n+1} \in S$ is drawn
according to $q(\cdot\mid s_{n},a_{n})$, and the game proceeds to stage $n+1$.

We extend the domain of $q$ and $(u^i)_{i \in I}$ to $\cup_{s \in S} \left(\{s\} \times \Delta(A(s))\right)$
in a multilinear fashion:
for every state $s \in S$ and every mixed action profile $\alpha \in \times_{i \in I}\Delta(A^i(s))$ we define
\begin{eqnarray*}
q(s,\alpha) := \sum_{a \in A(s)} \alpha[a] q(s,a),
\end{eqnarray*}
and
\begin{eqnarray*}
u^i(s,\alpha) := \sum_{a \in A(s)} \alpha[a] u^i(s,a),
\end{eqnarray*}
where $\alpha[a] := \prod_{i \in I} \alpha^i(a^i)$.

\subsection{Strategies and Payoffs}

A \emph{finite history of length $n$} is a sequence $h_n = (s_0,a_0,\cdots,s_n) \in \Lambda^n \times S$.
By convention, the set $\Lambda^0$ contains only the empty history.
Let $H := \cup_{n \geq 0} (\Lambda^{n} \times S)$ be the set of all finite histories
and $H^\infty := \Lambda^\infty$ be the set of \emph{plays}.
When $h \in H^\infty$ is a play and $n \geq 0$ we denote by $h_n \in \Lambda^{n} \times S$ the prefix of $h$ of length $n$.
The space $H^\infty$ together with the $\sigma$-algebra generated by all finite cylinders is a measure space.
We denote by $\calH(n)$ the algebra generated by the finite histories of length $n$.
We assume perfect recall.
Accordingly, a (behavior) \emph{strategy} of player $i$ is a function $\sigma^i : H \to \Delta(A^i)$ such that $\sigma^i(h_n) \in \Delta(A^i(s_n))$
for every finite history $h_n = (s_0,a_0,\cdots,s_n) \in H$.
Denote by $\Sigma^i$ the set of all strategies of player $i$,
by $\Sigma := \times_{i \in I} \Sigma^i$ the set of all strategy profiles,
and by $\Sigma^{-i} := \times_{j \neq i} \Sigma^j$ the set of all strategy profiles of all players except player $i$.

A class of simple strategies is the class of \emph{stationary strategies}.
Those are strategies in which the choice of a player at each stage depends only on the current state,
and not on previously visited states or on past choices of the players.
A stationary strategy of player $i$ can be identified with an element of
$\Sigmaistat := \times_{s \in S} \Delta(A^i(s)) \subset \dR^{\sum_{s \in S} |A^i(s)|}$,
and will be denoted $x^i = (x^i(s))_{s \in S}$.
A strategy profile $\sigma = (\sigma^i)_{i \in I}$ is \emph{stationary}
if for every player $i \in I$ the strategy $\sigma^i$ is stationary.
Denote by $\Sigmastat = \times_{i \in I} \Sigmaistat$ the set of stationary strategy profiles.

In Section~\ref{section:mdp} we will make use of the concept of \emph{general strategy} (see Mertens, Sorin, and Zamir, 2015),
which is a probability distribution over behavior strategies.
By Kuhn's Theorem (Kuhn, 1956) every general strategy is equivalent to a behavior strategy.

Every initial state $s_0 \in S$ and every strategy profile $\sigma = (\sigma^i)_{i \in I}\in \Sigma$
induce a probability distribution $\prob_{s_0,\sigma}$ over the set of plays $H^\infty$.
Denote the corresponding expectation operator by $\E_{s_0,\sigma}$.

For every discount factor $\lambda \in [0,1)$, every initial state $s_0$, and every strategy profile $\sigma$, the \emph{$\lambda$-discounted payoff}
is
\begin{equation}
\label{equ:payoff:discounted}
\gamma_\lambda^i(s_0;\sigma) := \E_{s_0,\sigma}\left[ (1-\lambda) \sum_{n=0}^\infty \lambda^{n} u^i(s_n,a_n)\right].
\end{equation}
For every integer $N \geq 0$, every initial state $s_0$, and every strategy profile $\sigma$, the \emph{$N$-stage payoff} is
\begin{equation*}
%\label{equ:payoff:n}
\gamma_N^i(s_0;\sigma) := \frac{1}{N}\E_{s_0,\sigma}\left[ \sum_{n=0}^{N-1} u^i(s_n,a_n)\right].
\end{equation*}

\section{The Modified Games: Definition and Summary of Results}
\label{section:definition}

The main subject of this paper is the modified game, which is an auxiliary normal-form game that corresponds to a given stochastic game
$\Gamma = (S,I,(A^i(s))_{s \in S, i \in I},(u^i)_{i \in I},q)$.
In this normal-form game, the set of players is $I$
and the set of strategies of each player $i \in I$ is $\Sigma^i$,
as in the stochastic game.
The payoff function of the modified game depends on four elements:
(a) an initial state $s_0 \in S$,
(b) a discount factor $\lambda \in [0,1)$,
(c) a collection $\vec\calD = (\calD^i)_{i \in I}$ of partitions of the set of states, one partition for each player,
and
(d) a collection of vectors $\vec c = (c^i)_{i \in I}$, where $c^i = (c^i(D))_{D \in \calD^i} \in \dR^{\calD^i}$ for each player $i \in I$.
In the next subsection we define the payoff function of the modified game.

\subsection{Definition of the modified game}

\begin{definition}
For every strategy profile $\sigma$, every discount factor $\lambda \in [0,1)$,
every initial state $s_0 \in S$,
every state $s \in S$,
and every action profile $a \in A(s)$,
the \emph{expected $\lambda$-discounted time} the play spends in state $s$ and the players play the action profile $a$ is
\[ t_\lambda(s_0,\sigma;s,a) := \E_{s_0,\sigma} \left[ (1-\lambda)\sum_{n=0}^\infty \lambda^n \mathbf{1}_{\{s_n=s,a_n=a\}}\right]. \]
We denote by $t_\lambda(s_0,\sigma) :=  (t_\lambda(s_0,\sigma;s,a))_{s \in S, a \in A(s)}$ the
\emph{state-action discounted time vector}.
\end{definition}

Fix an initial state $s_0 \in S$,
a discount factor $\lambda \in [0,1)$,
and a strategy profile $\sigma \in \Sigma$.
For every set of states $D \subseteq S$,
the \emph{$\lambda$-discounted time} that the play spends in $D$ is
\[ t_\lambda(s_0,\sigma;D) :=
\E_{s_0,\sigma}\left[ (1-\lambda)\sum_{n=0}^\infty \lambda^n \mathbf{1}_{\{s_n \in D\}} \right]
= \sum_{s \in D, a \in A(s)}t_\lambda(s_0,\sigma;s,a). \]
Let $i \in I$ be a player.
The \emph{$\lambda$-discounted (unnormalized) payoff} that player~$i$ receives while the play is in $D$ is
\begin{eqnarray}
\nonumber
U^i_\lambda(s_0,\sigma;D) &:=&
\E_{s_0,\sigma}\left[ (1-\lambda)\sum_{n=0}^\infty \lambda^n u^i(s_n,a_n) \mathbf{1}_{\{s_n \in D\}} \right]\\
&=& \sum_{s \in D, a \in A(s)}t_\lambda(s_0,\sigma;s,a)u^i(s,a).
\label{equ:105}
\end{eqnarray}
The quantity $\frac{U^i_\lambda(s_0,\sigma;D)}{t_\lambda(s_0,\sigma;D)}$ is the normalized $\lambda$-discounted payoff
of player~$i$ during visits to $D$.

Given a partition $\calD^i$ of the set of states and a vector $c^i = (c^i(D))_{D \in \calD^i} \in \dR^{\calD^i}$
define a new payoff function ${\widehat \gamma}_\lambda^i(s_0,\cdot;\calD^i,c^i) : \Sigma \to \dR$ over the set of strategy profiles by
\begin{eqnarray}
\label{equ:payoff:2}
{\widehat \gamma}_\lambda^i(s_0,\sigma;\calD^i,c^i) &:=&
\sum_{D \in \calD^i}
\min\left\{U^i_\lambda(s_0,\sigma;D), t_\lambda(s_0,\sigma;D)\cdot  c^i(D)\right\}\\
&=&
\sum_{D \in \calD^i}
t_\lambda(s_0,\sigma;D)\min\left\{\tfrac{U^i_\lambda(s_0,\sigma;D)}{t_\lambda(s_0,\sigma;D)},  c^i(D)\right\},
\ \ \ \forall \sigma \in \Sigma,
\nonumber
\end{eqnarray}
where by convention $\tfrac{0}{0}=1$.
Thus, to calculate the payoff $\widehat \gamma_\lambda^i(s_0,\sigma;\calD^i,c^i)$,
we take the normalized $\lambda$-discounted payoff during the visits to each element $D$ of $\calD^i$,
and, if this quantity is higher than $c^i(D)$, we lower it to $c^i(D)$.
We then sum up the quantities that we obtained for all elements of $D$,
after multiplying each one by the $\lambda$-discounted time the play spends in $D$.
Accordingly, with the new payoff function $c^i(D)$ is the maximal amount that player~$i$ can receive during the visits to $D$.

For every partition $\calD^i$ of the set of states we have
\begin{equation}
\label{equ:106}
\gamma^i_\lambda(s_0,\sigma) = \sum_{D \in \calD^i} U^i_\lambda(s_0,\sigma;\calD^i),
\ \ \ \forall \lambda \in [0,1), \forall s_0 \in S, \forall i \in I, \forall \sigma \in \Sigma.
\end{equation}
Consequently, by Eq.~(\ref{equ:payoff:2})
\begin{equation}
\label{equ:inequality:1}
{\widehat \gamma}_\lambda^i(s_0,\sigma;\calD^i,c^i) \leq \gamma_\lambda^i(s_0,\sigma),
\ \ \ \forall \lambda \in [0,1), \forall s_0 \in S, \forall i \in I, \forall \sigma \in \Sigma.
\end{equation}
Since payoffs are bounded by 1, if $c^i(D) \geq 1$ for every element $D \in \calD^i$,
then there is an equality in Eq.~(\ref{equ:inequality:1}).

\begin{definition}
\label{definition:modified}
Let $\lambda \in [0,1)$ be a discount factor,
let $s_0 \in S$ be a state,
let $\vec\calD = (\calD^i)_{i \in I}$ be a collection of partitions of the set of states,
and let $\vec c = (c^i)_{i \in I} \in \times_{i \in I}\dR^{\calD^i}$.
The \emph{modified games} $\widehat\Gamma_\lambda(s_0;\vec \calD, \vec c)$ is the normal-form game
$(I,\Sigma, (\widehat \gamma^i_\lambda(s_0,\cdot;\calD^i,c^i))_{i \in I})$.
\end{definition}

\begin{remark}[Alternative definitions of the modified game]
\label{remark:1}
The minimum in Eq.~(\ref{equ:payoff:2}) that defines the payoff in the modified game
is between the discounted payoff during the visits to an element $D$ of $\calD^i$
and the normalized cut-off.
Several alternative definitions come to mind:
\begin{enumerate}
\item   The minimum could have been taken in each stage separately, instead of dividing the stages according to the element of $\calD^i$
that contains the current state.
\item   The minimum could have been taken for each play path separately.
\item   The minimum could have been taken for each visit to an element of $\calD^i$ separately,
instead of aggregating all visits to the same element.
\end{enumerate}
Each alternative definition will satisfy some of the results that we prove in the sequel, but not all of them.
\end{remark}

\begin{remark}[The modified game in absorbing games]
When specialized to absorbing games and the partition $\calD^i$ that contains only singletons, for every player $i \in I$,
the modified game defined by Solan (1999) coincides with alternative 1 in Remark~\ref{remark:1}.
Another variant of the modified game for absorbing games was presented in Solan and Vohra (2002).
%In Section~\ref{variants} we will present several variants of the modified game for multiplayer stochastic games.
Unlike Solan (1999) and Solan and Vohra (2002), where the modified game was an absorbing game,
just like the original game,
here the modified game is a normal-form game.
\end{remark}

\begin{remark}[On the discount factor]
Definition~\ref{definition:modified} assumes that all players share the same discount factor.
Our results continue to hold if each player has a different discount factor.
\end{remark}

\begin{remark}[The modified game of zero-sum games]
It is worth noting that when the original stochastic game is a two-player zero-sum game,
the modified game is no longer a zero-sum game.
\end{remark}

\subsection{Summary of Results}

Since the paper is long and presents many results,
some on the modified game and some that are used in the proofs of the main results,
we end this section by providing a list of the main results concerning the modified game.

\begin{itemize}
\item   The modified game $\widehat\Gamma_\lambda(s_0;\vec \calD,\vec c)$ admits an equilibrium.
Moreover, there is an equilibrium in which all players use stationary strategies (Theorem~\ref{theorem:eq:2}).
\item
The function that assigns to each discount factor $\lambda$ the set of equilibria in stationary strategies of the modified game
$\widehat\Gamma_\lambda(s_0;\vec \calD,\vec c)$
is semi-algebraic (Theorem~\ref{theorem:eq:semialgebraic}).
\item
Denote by $\underline v^i_\lambda(s)$ the $\lambda$-discounted max-min value of player~$i$
in the stochastic game when the initial state is $s$.
If the partition $\calD^i$ satisfies a certain property (see Definition~\ref{def:property:p}),
and if $c^i(D) \geq \lim_{\lambda \to 1} \underline v^i_\lambda(s)$ for every element $D \in \calD^i$ and every state $s \in D$,
then the limit of the max-min value of the modified game as the discount factor goes to 1 exists and is equal to $\lim_{\lambda \to 1} \underline v^i_\lambda(s_0)$
(Theorem~\ref{theorem:5}).
\item
Denote by $\overline v^i_\lambda(s)$ the $\lambda$-discounted min-max value of player~$i$
in the stochastic game when the initial state is $s$.
If the partition $\calD^i$ satisfies a certain property (see Definition~\ref{def:property:ptag}),
and if $c^i(D) \geq \lim_{\lambda \to 1}  \overline v^i_\lambda(s)$ for every element $D \in \calD^i$ and every state $s \in D$,
then the limit of the min-max value of the modified game as the discount factor goes to 1 exists and is equal to $\lim_{\lambda \to 1}  \overline v^i_\lambda(s_0)$
(Theorem~\ref{theorem:6}).
\item
The min-max value in stationary strategies of the modified game is at least the min-max value of the modified game
and at most the discounted min-max value in the original game (Theorem~\ref{theorem:minmax:comparison}).
\item
The max-min value in stationary strategies of the modified game is at most the max-min value of the modified game (Theorem~\ref{theorem:maxmin:comparison}),
and the difference may not vanish as the discount factor goes to 1 (Example~\ref{example:big match}).
\end{itemize}

\section{The Modified Stochastic Game: Equilibrium}
\label{section:equilibrium}

Using standard arguments one can show that the modified game admits an equilibrium.
Indeed, the space of pure strategies is compact in the product topology,
hence the space of mixed strategies is compact in the weak-* topology.
Moreover, the modified payoff function, defined on the space of profiles of mixed strategies, is continuous in the weak-* topology
and concave in each of its coordinates.
By, e.g., Schauder's fixed point theorem, an equilibrium exists.

The payoff function in the modified game is neither multilinear nor quasiconcave when restricted to stationary strategies,
hence it is not clear that the game admits an equilibrium in stationary strategies.
Nevertheless we will prove in this section that the modified game admits an equilibrium in stationary strategies.

\begin{theorem}
\label{theorem:eq:2}
For every initial state $s_0 \in S$,
every discount factor $\lambda \in [0,1)$,
every collection $\vec\calD=(\calD^i)_{i \in I}$ of partitions of the set of states,
and every vector of cutoffs $\vec c = (c^i)_{i \in I} \in \times_{i \in I}\dR^{\calD^i}$,
the modified game ${\widehat\Gamma}_\lambda(s_0;\vec \calD,\vec c)$ admits an equilibrium in stationary strategies.
\end{theorem}

When particularized to absorbing games, Theorem~\ref{theorem:eq:2} coincides with Step~1 in the proof of Theorem~4.5 in Solan (1999).
In Section~\ref{section:mdp} we present some technical tools that are needed to prove Theorem~\ref{theorem:eq:2}.
The proof of the theorem together with additional results appear in Section~\ref{subsection:2}.

\subsection{A Result on Stationary Strategies in MDP's}
\label{section:mdp}

Below we prove the existence of an equilibrium in the modified game by using a fixed point theorem.
To this end we will consider the best response of one player
to fixed strategies of the other players.
When fixing the stationary strategies of all players but one,
the game is reduced to a Markov decision problem.
The best-response set-valued function in the modified game does not have convex values,
hence we will need a new tool to prove the existence of a fixed point, which we describe in this section.

A \emph{Markov decision process} (MDP) is a stochastic game with a single player.
For notational convenience we denote the unique player by $i$.

%A strategy $\sigma^i$ is \emph{stationary at state $s$} if $\sigma^i(h) = \sigma^i(h')$ for every two finite histories $h,h' \in H_s$.
%Strategies that are stationary at some states may be nonstationary in other states.
%A strategy is \emph{stationary} if it is stationary at all states.

%For every strategy $\sigma^i$ and every finite history $h \in H$ denote by $\sigma_h^i$ the continuation strategy that occurs once history $h$ is realized.
%Note that in this case the initial state is assumed to be the last state of $h$.

We say that the two strategies $\sigma^i$ and ${\sigma'}^i$ are \emph{$\lambda$-equivalent} at the initial state $s_0$
if the two strategies induce the same state-action discounted time vector, that is, $t_\lambda(s_0,\sigma^i) = t_\lambda(s_0,{\sigma'}^i)$.
As the following result states,
for every strategy there exists a $\lambda$-discounted equivalent stationary strategy at any given initial state.
\begin{theorem}[Altman, 1999, Corollary 10.1]
\label{proposition:stationary:mdp}
For every discount factor $\lambda \in [0,1)$, every initial state $s_0 \in S$, and every
strategy $\sigma^i \in \Sigma^i$ there exists a stationary strategy $x^i \in \Sigmaistat$ that is $\lambda$-equivalent to $\sigma^i$ at $s_0$.
\end{theorem}

The following result states that the stationary strategy $x^i$ that is $\lambda$-equivalent to $\sigma^i$ at $s_0$ is unique,
up to its definition in states that are never visited.
\begin{lemma}
\label{lemma:uniqueness}
Let $\lambda \in [0,1)$ be a discount factor, let $s_0 \in S$ be a state,
let $x^i$ and ${x'}^i$ be two stationary strategies in $\Sigmaistat$ that are $\lambda$-discounted equivalent at $s_0$,
let $s \in S$, and let $a \in A^i(s)$.
If $t_\lambda(s_0,x^i;\{s\}) > 0$ then $x^i(s;a) = {x'}^i(s;a)$.
\end{lemma}

\begin{proof}
The stationarity of $x^i$ implies that
\[ x^i(s;a) = \frac{t_\lambda(s_0,x^i;s,a)}{\sum_{a' \in A^i(s)} t_\lambda(s_0,x^i;s,a')} = \frac{t_\lambda(s_0,x^i;s,a)}{t_\lambda(s_0,x^i;\{s\})}, \]
whenever the denominator is positive,
and a similar equality holds for ${x'}^i$.
Since $x^i$ and ${x'}^i$ are $\lambda$-equivalent at $s_0$,
either both denominators (for $x^i$ and for ${x'}^i$) are positive, or both are equal to 0.
The result follows.
\end{proof}

\bigskip

We will need below the following observation that strengthens Theorem~\ref{proposition:stationary:mdp} and that is stated without a proof.
It says that if $x^i$ is a stationary strategy that is $\lambda$-discounted equivalent to $\sigma^i$,
then the probability that $x^i$ plays the action $a$ at state $s$ cannot exceed the maximal probability
that $\sigma^i$ plays $a$ in some visit to $s$,
and cannot fall below the minimal probability
that $\sigma^i$ plays $a$ in some visit to $s$.

To formally state the result, for every state $s \in S$ we denote by $H_s$ the set of all finite histories that end at $s$:
\[ H_s := \{(s_0,a_0,\cdots,s_n) \in H \colon s_n = s\}. \]

\begin{lemma}
\label{proposition:stationary:mdp2}
Let $\lambda \in [0,1)$ be a discount factor, let $s_0 \in S$ be a state, and let $\sigma^i \in \Sigma^i$ be a strategy.
Let $x^i \in \Sigmaistat$ be a stationary strategy that is $\lambda$-equivalent to $\sigma^i$ at $s_0$.
Then
\[ \inf_{h \in H_s} \sigma^i(h;a) \leq x^i(s;a) \leq \sup_{h \in H_s} \sigma^i(h;a). \]
\end{lemma}

It is well known that the set of stationary optimal strategies in an MDP is convex.
The following lemma implies in particular a weaker version of this result, namely that the set of optimal stationary strategies is monovex.
Below we will use the lemma for a modified game, in which the set of stationary optimal strategies is not convex.
\begin{lemma}
\label{lemma:monotone}
Let $s_0 \in S$ be a state and let $x^i$ and ${x'}^i$ be two stationary strategies in $\Sigmaistat$.
Let $\sigma^i_\alpha := [\alpha(x^i),(1-\alpha)({x'}^i)]$ be the general strategy that follows
$x^i$ with probability $\alpha$ and ${x'}^i$ with probability $1-\alpha$.
For every $\alpha \in [0,1]$ let $x^i_\alpha$ be a stationary strategy
that is $\lambda$-equivalent to the strategy $\sigma^i_\alpha$ at the initial state $s_0$ (see Proposition~\ref{proposition:stationary:mdp}),
so that
\[ t_\lambda(s_0,x^i_\alpha) = \alpha t_\lambda(s_0,x^i) + (1-\alpha) t_\lambda(s_0,{x'}^i). \]
Then the function $\alpha \mapsto x^i_\alpha(s;a)$ is
monotone for every state $s \in S$ for which $\sum_{a' \in A^i(s)} t_\lambda(s_0,x^i;s,a') + \sum_{a' \in A^i(s)} t_\lambda(s_0,{x'}^i;s,a') > 0$
and for every action $a \in A^i(s)$.
\end{lemma}

\begin{proof}
Note that if $\alpha = \beta \alpha' + (1-\beta)\alpha''$ then $x^i_\alpha$ is $\lambda$-equivalent at $s_0$
to the general strategy $[\beta(x^i_{\alpha'}),(1-\beta)(x^i_{\alpha''})]$.
Together with Lemma~\ref{proposition:stationary:mdp2} this implies that the function $\alpha \mapsto x^i_\alpha(s;a)$ is monotone
for every state $s \in S$ for which $t_\lambda(s_0,x^i;\{s\}) + t_\lambda(s_0,{x'}^i;\{s\}) > 0$
and for every action $a \in A^i(s)$.
The continuity of this function follows from Lemma~\ref{lemma:uniqueness}
and from the continuity of the function $x^i \mapsto t_\lambda(s_0,x^i)$.
\end{proof}

\bigskip

As a corollary we deduce the following.

\begin{corollary}
\label{corollary:36}
Let $g : \Sigma^i \to \dR$ be a real-valued function that satisfies the following conditions:
\begin{enumerate}
\item   The function $g$ depends on its parameter only through its state-action discounted time vector:
there are $s_0 \in S$, $\lambda \in (0,1)$, and a continuous function
$f : \Delta(S \times A) \to \dR$ such that $g(s_0,\sigma^i) = f(t_\lambda(s_0,\sigma^i))$ for every $\sigma^i \in \Sigma^i$.
\item   The function $g$ is quasiconcave:
For every $\sigma^i,\sigma'^i \in \Sigma^i$ and every $\beta \in [0,1]$ we have
\[ g([\beta(\sigma^i),(1-\beta)(\sigma'^i)]) \geq \beta g(\sigma^i)+(1-\beta)g(\sigma'^i). \]
\end{enumerate}
Then the set $\argmax_{\sigma^i \in \Sigma^i} g(\sigma^i)$ of maximizers of $g$ is a closed monovex set.
\end{corollary}

\begin{proof}
The set $\argmax_{\sigma^i \in \Sigma^i} g(\sigma^i)$ is closed since the function $f$ is continuous.
The set is monovex by Lemma~\ref{lemma:monotone}.
\end{proof}

\subsection{Stationary Equilibria in the Modified Game}
\label{subsection:2}

Because for every set of states $D \subseteq S$
the functions $t_\lambda(s_0,\cdot;D)$ and $U^i_\lambda(s_0,\cdot;D)$ are continuous over $\Sigma$
and because the minimum of two continuous function is a continuous function,
it follows that the payoff functions in the modified game is continuous over the space of strategy profiles $\Sigma$.
This observation is summarized by the following lemma.

\begin{lemma}
\label{lemma:continuous}
For every $\lambda \in [0,1)$,
every initial state $s_0 \in S$,
every player $i\in I$,
every partition $\calD^i$ of the set of states,
and every vector $c^i \in \dR^{\calD^i}$,
the function $\sigma \mapsto \widehat \gamma_\lambda^i(s_0,\sigma;\calD^i,c^i)$ is continuous.
\end{lemma}

We will now show that in the modified game,
when the other players play a stationary strategy, player~$i$ has a stationary best response.
\begin{lemma}
\label{lemma:br}
For every $\lambda \in [0,1)$, every initial state $s_0 \in S$,
every player $i\in I$,
every partition $\calD^i$ of the set of states,
every vector $c^i \in \dR^{\calD^i}$,
and every stationary strategy profile $x^{-i} \in \Sigma^{-i}_{\hbox{\small{stat}}}$ of the other players,
there is a stationary strategy $x^i \in \Sigma^i_{\hbox{\small{stat}}}$ that maximizes player~$i$'s payoff in the modified game:
\begin{equation}
\label{equ:500}
{\widehat\gamma}^i_\lambda(s_0,x^i,x^{-i};\calD^i,c^i) = \max_{\sigma^i \in \Sigma^i} {\widehat\gamma}^i_\lambda(s_0,\sigma^i,x^{-i};\calD^i,c^i).
\end{equation}
\end{lemma}

\begin{proof}
Fixing a stationary strategy profile $x^{-i}$ of the other players, the decision problem of player~$i$ in the modified game becomes a Markov decision problem,
albeit with the modified payoff function.

The set of all strategies of player~$i$ is compact and the payoff function $\sigma^i \mapsto {\widehat\gamma}^i_\lambda(s_0,\sigma^i,x^{-i};\calD^i,c^i)$
is continuous (Lemma~\ref{lemma:continuous}), hence the maximum on the right-hand side of Eq.~(\ref{equ:500})
is attained by some strategy $\sigma^i\in \Sigma^i$.
By Proposition~\ref{proposition:stationary:mdp}, at the initial state $s_0$ the strategy $\sigma^i$ induces the same state-action discounted time vector
as some stationary strategy $x^i \in \Sigma^i_{\small{\hbox{stat}}}$.
By Eq.~(\ref{equ:105}) the payoff ${\widehat\gamma}^i_\lambda(s_0,\sigma^i,x^{-i};\calD^i,c^i)$ depends on $\sigma^i$ only through
the state-action $\lambda$-discounted time vector,
and therefore
${\widehat\gamma}^i_\lambda(s_0,\sigma^i,x^{-i};\calD^i,c^i) = {\widehat\gamma}^i_\lambda(s_0,x^i,x^{-i};\calD^i,c^i)$, and the claim follows.
\end{proof}

\bigskip

In the proof of Theorem~\ref{theorem:eq:2} we will use sets in a Euclidean space that satisfy a condition that resembles convexity;
specifically, sets that satisfy the following condition: for every two points in the set there is a continuous path in the set
that connects the two points and that is monotone in each coordinate.
Such sets, which were introduced in Buhuvsky, Solan, and Solan (2016), are called \emph{monovex}.

\begin{definition}
A set $X \subseteq \dR^d$ is \emph{monovex}
if for every $x,y \in X$ there is a continuous path $f : [0,1] \to X$ that satisfies the following properties:
\begin{enumerate}
\item   $f(0) = x$ and $f(1) = y$.
\item   $f_i(t) : [0,1] \to \dR$ is a monotone function (nondecreasing or nonincreasing) for every $i\in \{1,2,\ldots,d\}$.
\end{enumerate}
\end{definition}

The main property of monovex sets that we will need is that every closed monovex set is contractible (see Buhuvsky, Solan, and Solan, 2016).
We are now ready to prove Theorem~\ref{theorem:eq:2}.

\bigskip

\begin{proof}[Proof of Theorem~\ref{theorem:eq:2}]
Define a set-valued function $F : X \to X$ as follows.
For every player $i \in I$ and every stationary strategy profile $x \in X$,
\[ F^i(x) := \left\{ y^i \in \Sigma^i_{\small{\hbox{stat}}} \colon
{\widehat\gamma}^i_\lambda(s_0,y^i,x^{-i};\calD^i,c^i) = \max_{\sigma^i \in \Sigma^i} {\widehat\gamma}^i_\lambda(s_0,\sigma^i,x^{-i};\calD^i,c^i)\right\}. \]
By Lemma~\ref{lemma:br} the set $F^i(x)$ is nonempty for every player $i \in I$
and every stationary strategy profile $x \in \Sigma_{\hbox{stat}}$.
By Corollary~\ref{corollary:36} this set is monovex,
and therefore by Buhovsky, Solan, and Solan (2016) it is contractible.
Lemma~\ref{lemma:continuous} implies that the graph of $F$ is closed.
By the Eilenberg and Montgomery (1946) fixed point theorem
the set-valued function $F$ has a fixed point, which is a stationary equilibrium in the modified game.
\end{proof}

\bigskip

As the following example shows, the equilibrium in stationary strategies
that is guaranteed to exist by Theorem~\ref{theorem:eq:2} may depend on the initial state $s_0$,
even if the set of players includes a single player.

\begin{example}
\label{example:1}
Consider the following stochastic game with four states, which depends on two parameters, $y \in \dR$ and $p \in [0,1]$,
and is graphically given in Figure \arabic{figurecounter};
in this figure, the payoff appears on the left-hand side of each entry, while the transition appears on the right-hand side.
\begin{itemize}
\item   There is a single player: $I = \{i\}$.
\item   There are four states $S = \{s^0,s^1,s^2,s^3\}$.
States $s^2$ and $s^3$ are absorbing and yield payoffs 2 and 3, respectively.
\item   In state $s^0$ the player has a single action. His payoff is $y$ and the play moves to state $s^1$.
\item   In state $s^1$ the player has two actions, $T$ and $B$. The payoff in this state and the transition
are given in  Figure \arabic{figurecounter}.
\end{itemize}

\bigskip

\centerline{
\begin{picture}(40,65)(0,-20)
\put( 0,0){\numbercellongb{$y$}{$^{s^1}$}}
\put( 20,-20){$s^0$}
\end{picture}
\ \ \
\begin{picture}(110,65)(-10,-20)
\put(-10, 8){$B$}
\put(-10,28){$T$}
\put( 50,-20){$s^1$}
\put( 0,20){\numbercellonga{$0$}{$^{[(1-p)(s^1),p(s^2)]}$}}
\put( 0,0){\numbercellonga{$-1$}{$^{[(1-p)(s^1),p(s^3)]}$}}
\end{picture}
\ \ \
\begin{picture}(40,65)(0,-20)
\put( 0,0){\numbercellongb{$2$}{$^{s^2}$}}
\put( 20,-20){$s^2$}
\end{picture}
\ \ \
\begin{picture}(40,65)(0,-20)
\put( 0,0){\numbercellongb{$3$}{$^{s^3}$}}
\put( 20,-20){$s^3$}
\end{picture}
}

\centerline{Figure \arabic{figurecounter}: The game in Example~\ref{example:1}.}

\bigskip

Let $\calD^i = \bigl\{ \{s^0,s^1\}, \{s^2\}, \{s^3\} \bigr\}$,
and $c^i = (c^i(D))_{D \in \calD^i}$ be the vector
\[ c^i(\{s^0,s^1\}) = 0, \ \ \ c^i(\{s^2\}) = 2, \ \ \ c^i(\{s^3\}) = 3. \]
Thus, the payoff may be lowered only in the set $\{s^0,s^1\}$.
We will show that for a proper choice of the parameters $\lambda$, $p$, and $y$,
the optimal strategy in the modified game ${\widehat\Gamma}_\lambda(s_0; \calD^i, c^i)$ depends on the initial state:
when the initial state is $s^1$ the optimal strategy consists of playing $T$ in state $s^1$,
while when the initial state is $s^0$ the optimal strategy consists of playing $B$ in state $s^1$.

The intuition for this result is as follows:
When the initial state is $s^1$ the payoff while the play is in $\{s^0,s^1\}$ is nonpositive,
hence the minimum with $c^i(\{s^0,s^1\})$ has no effect on the payoff in the modified game,
and the modified game is reduced to a standard Markov decision problem.
The parameters $p$ and $\lambda$ will be chosen so that the optimal strategy is to play $T$ in this case.
Suppose now that the initial state is $s^0$.
Since $y$ is positive, the player is better off choosing $B$ at state $s^1$ as long as his total discounted payoff before absorption is nonnegative.
This increases the probability to be absorbed in the better absorbing state $s^3$,
while not affecting the modified payoff that he receives
for the stages in which the play visits the set $\{s^0,s^1\}$.

We turn to the formal calculations.
As mentioned before, when the initial state is $s^1$ the modified game is a standard Markov decision problem,
hence there is an optimal strategy which is pure and stationary.
Since at state $s^1$ the play is absorbed at every stage with probability $p$,
the total discounted weight of the absorbing state on the payoff is
\[ \lambda p + (1-p) \lambda^2 p + (1-p)^2 \lambda^3 p + \cdots = \frac{\lambda p}{1-\lambda(1-p)}. \]
It follows that when the initial state is $s^1$, the stationary strategy $T$ is the unique optimal strategy as soon as
\[ 2  \frac{\lambda p}{1-\lambda(1-p)} > -1 \left( 1 -  \frac{\lambda p}{1-\lambda(1-p)}\right) + 3  \frac{\lambda p}{1-\lambda(1-p)}, \]
which solves to
\begin{equation*}
%\label{equ:112}
1 > \lambda(1+p).
\end{equation*}

Suppose now that the initial state is $s^0$.
We choose $y$ so that under the stationary strategy that plays $B$ in $s^0$, the discounted payoff when the play is in $\{s^0,s^1\}$ is $0$.
That is,
\[ (1-\lambda)y - \lambda\left(1 - \frac{\lambda p}{1-\lambda(1-p)}\right) = 0, \]
which solves to
\begin{equation}
\label{equ:113}
y = \frac{\lambda}{1-\lambda(1- p)}.
\end{equation}
This implies that when $y$, $\lambda$, and $p$ satisfy Eq.~(\ref{equ:113}),
under \emph{every} strategy, the modified payoff in the set $\{s^0,s^1\}$ is 0,
hence the optimal strategy is the one that maximizes the expected absorbing payoff,
which is the stationary strategy that plays $B$ at state $s^1$.
\addtocounter{figurecounter}{1}
\end{example}

%We nevertheless show that the game $\widehat{\widehat\Gamma}_\lambda(s_0;\calD,c)$ always has an equilibrium.
%Moreover, there is such an equilibrium in which the play at each stage $n$ depends on
%(a) the current state $s_n$, and
%(b) the first state in the set $D(s_n)$ that was visited during the current visit to $D(s_n)$.

We end this section by extending the semi-algebraic property of the set of discounted equilibria of stochastic games to the modified game.
Denote by $E_\lambda(s_0;\vec \calD,\vec c)$ the set of all stationary equilibria of the game $\widehat\Gamma_\lambda(s_0;\vec \calD,\vec c)$.
Because for every initial state $s_0 \in S$, every set of states $D \subseteq S$, and every player $i \in I$ the functions
$(\lambda,x) \mapsto t_\lambda(s_0,x;D)$
and $(\lambda,x) \mapsto U^i_\lambda(s_0,x;D)$
are semi-algebraic,
where $x$ ranges over all stationary strategy profiles,
we obtain the following result.

\begin{theorem}
\label{theorem:eq:semialgebraic}
For every initial state $s_0 \in S$,
every collection $\vec\calD = (\calD^i)_{i \in I}$ of partitions of the set of states,
and every collection of cutoffs $\vec c = (c^i)_{i \in I} \in \times_{i \in I}\dR^{\calD^i}$
the set-valued function $\lambda \mapsto E_\lambda(s_0;\vec\calD,\vec c)$ is semi-algebraic.
\end{theorem}

\section{The Max-Min Value in Stochastic Games}
\label{section:max:min}

Our next goal is to study the max-min value in the modified game.
To this end we recall the concept of max-min value in the original stochastic game
and a result of Neyman (2003) on the uniform max-min value in stochastic game.

The \emph{$\lambda$-discounted max-min value of player $i$ at the initial state $s_0$} in the original game is given by
\begin{equation}
\label{equ:maxmin}
\underline v_\lambda^i(s_0) := \max_{\sigma^i \in \Sigma^i}\min_{\sigma^{-i} \in \Sigma^{-i}} \gamma_\lambda^i(s_0;\sigma^i,\sigma^{-i}).
\end{equation}
This is the maximal amount that player $i$ can guarantee if the other players get to know his strategy and try to minimize his payoff.
Because the $\lambda$-discounted payoff (see Eq.~(\ref{equ:payoff:discounted})) is a continuous function of the strategies of the players,
the maximum and minima in Eq.~(\ref{equ:maxmin}) are attained.
It is well known that the maximum and minima in Eq.~(\ref{equ:maxmin}) are attained by stationary strategies.
Moreover, the function $\lambda \mapsto \underline v_\lambda^i(s_0)$ is semi-algebraic
(see Bewley and Kohlberg (1976) and Neyman (2003)),
and therefore the limit
\[ \underline v_1^i(s_0) := \lim_{\lambda \to 1} \underline v_\lambda^i(s_0) \]
exists for every player $i \in I$ and every initial state $s_0 \in S$.
The quantity $\underline v_1^i(s_0)$ is called the \emph{uniform max-min value} of player $i$ at the initial state $s_0$.

For every two bounded stopping times $\tau < \tau'$, the expected average payoff between stages $\tau$ and $\tau'$ is
\[ \gamma^i(s_0,\sigma;\tau,\tau') := \frac{\E_{s_0,\sigma}[\sum_{n=\tau}^{\tau'} r(s_n,a_n) \mid \calH(\tau)]}{\E_{s_0,\sigma}[\tau - \tau' \mid \calH(\tau)]}. \]
Note that $\gamma^i(s_0,\sigma;\tau,\tau')$ is a random variable that is measurable according to the information at stage $\tau$.

\begin{definition}
\label{defin:mn:maxmin}
Let $\ep > 0$ and let $i \in I$ be a player.
A strategy $\sigma^i$ is a \emph{uniform $\ep$-max-min} strategy of player~$i$
if there exist $\lambda_0 \in [0,1)$ and $N_0 \in \dN$ such that the following holds
for every initial state $s_0 \in S$, every strategy profile $\sigma^{-i} \in \Sigma^{-i}$ of the other players,
every discount factor $\lambda \in (\lambda_0,1)$,
and every bounded stopping time $\tau$ that satisfies $\E_{s_0,\sigma^i,\sigma^{-i}}[\tau] \geq N_0$:
\begin{eqnarray}
\label{equ:mn:1}
\gamma^i_\lambda(s_0,\sigma^i,\sigma^{-i}) &\geq& \underline v^i_1(s_0) - \ep,\\
\label{equ:mn:2}
\gamma^i(s_0,\sigma^i,\sigma^{-i};0,\tau) &\geq& \underline v^i_1(s_0) - \ep.
\end{eqnarray}
\end{definition}

\begin{remark}[On the definition of uniform $\ep$-max-min strategies]
The standard definition of uniform $\ep$-max-min strategies requires Eq.~(\ref{equ:mn:2}) only for constant stopping times $\tau$.
We will need in the sequel the stronger version that we presented here.
Sometimes a uniform $\ep$-max-min strategy is required to satisfy that the expected long-run average payoff is at least
$\underline v^i_1(s_0) - \ep$ for every strategy profile of the other players (see Mertens and Neyman (1981) and Neyman (2003)).
Our results are not affected by adding this requirement.
\end{remark}

For every player $i \in I$, every strategy $\sigma^i \in \Sigma^i$,
and every finite history $h_n \in H$,
denote by $\sigma^i_{h_n}$ the strategy profile $\sigma$ conditioned on the history $h_n$, that is,
\[ \sigma^i_{h_n}(\widehat h_m) := \sigma^{i}(s_0,a_0,\cdots,s_{n-1},a_{n-1},\widehat s_0,\widehat a_0,\widehat s_1,\widehat a_1,\cdots,\widehat s_m), \ \ \
\forall \widehat h_m = (\widehat s_0,\widehat a_0,\cdots,\widehat s_m) \in H. \]

\begin{definition}
\label{definition:preserving}
Let $\ep > 0$ and let $i \in I$ be a player.
A strategy $\sigma^i$ of player~$i$ is \emph{$\ep$-max-min preserving} if
for every strategy profile $\sigma^{-i} \in \Sigma^{-i}$,
every play $h \in H^\infty$,
and every two bounded stopping times $\tau < \tau'$ we have
\begin{equation}
\label{equ:mn:5}
\E_{s_\tau,\sigma^i_{h_\tau},\sigma^{-i}_{h_\tau}}[\underline v^i_1(s_{\tau'})] \geq \underline v^i_1(s_\tau) -  \ep.
\end{equation}
\end{definition}

A uniform $\ep$-max-min strategy need not be $\ep$-max-min preserving,
and an $\ep$-max-min preserving strategy need not be a uniform $\ep$-max-min strategy.

Neyman (2003) adapted the construction of Mertens and Neyman (1981) and constructed for every $\ep > 0$ a uniform $\ep$-max-min strategy
that is also $\ep$-max-min preserving.
\begin{theorem}[Neyman, 2003]
\label{theorem:mn}
For every $\ep > 0$
and every player $i \in I$,
there is a uniform $\ep$-max-min strategy $\sigma^i$ that is also $\ep$-max-min preserving.
\end{theorem}

\subsection{Markovian $\ep$-Max-Min Strategies}

A strategy of a player is Markovian w.r.t.~a partition if the mixed action played at each stage
depends only on the play during the current visit to the element of the partition that contains the current state.
Formally,

\begin{definition}
Let $i \in I$ be a player
and let $\calD$ be a partition of the set of states.
A strategy $\sigma^i \in \Sigma^i$ is \emph{$\calD$-Markovian} if $\sigma^i_{h_n} = \sigma^i_{\widetilde h_{\widetilde n}}$
for every two finite histories $h_n = (s_0,a_0,\cdots,s_n)$ and $\widetilde h_{\widetilde n} = (\widetilde s_0,\widetilde a_0,\cdots,\widetilde s_{\widetilde n})$
that satisfy
\begin{itemize}
\item   The two histories end at the same state: $s_n = \widetilde s_{\widetilde n}$.
\item   In stage $n$ (resp.~in stage $\widetilde n$) the finite history $h_n$ (resp.~$\widetilde h_{\widetilde n}$)
enters a new element of $\calD$,
that is, $D(s_{n-1}) \neq D(s_n)$ and
$D(\widetilde s_{\widetilde n-1}) \neq D(\widetilde s_{\widetilde n})$,
where $D(s)$ is the element of the partition $\calD$ that contains $s$, for every state $s \in S$,

\end{itemize}
\end{definition}

One naive way to define a $\calD$-Markovian strategy from a strategy $\sigma^i$ is the following:
whenever the play enters an element of $\calD$,
player~$i$ forgets past play and restarts playing $\sigma^i$.
We will denote this strategy $\sigma^i_{\calD}$ and define it now formally.

Given a partition $\calD$ of the set of states,
let $(\tau_k^\calD)_{k \geq 0}$ be the sequence of stopping times that indicates when the play moves from one element of the partition $\calD$
to another element of the partition:
\begin{eqnarray*}
\tau^\calD_{0} &:=& 0,\\
\tau_k^\calD &:=& \min\left\{n > \tau^\calD_{k-1} \colon D(s_n) \neq D(s_{n-1})\right\}, \ \ \ k \in \dN,
\end{eqnarray*}
where the minimum of an empty set is $+\infty$.
The stages between stage $\tau_k^\calD$ and $\tau_{k+1}^\calD-1$ are called the a {$\calD$-run}.
For every $n \geq 0$ let $k(\calD;n)$ be the number of the $\calD$-run that contains stage $n$;
that is, it is the unique nonnegative integer $k$ that satisfies $\tau^{\calD}_{k} \leq n < \tau^{\calD}_{k+1}$.

Denote by $\varphi(h_n)$ the play along $h_n$ since the last switch of an element in $\calD$:
\[ \varphi(h_n) :=
(s_{\tau^\calD_{k(\calD,n)}(h_n)},a_{\tau^\calD_{k(\calD,n)}(h_n)},s_{\tau^\calD_{k(\calD,n)}(h_n)+1},a_{\tau^\calD_{k(\calD,n)}(h_n)+1},\cdots,s_n). \]
The strategy $\sigma^i_{\calD}$ that is defined by
\[ \sigma^i_{\calD}(h_n) := \sigma^i(\varphi(h_n)), \ \ \ \forall h_n \in H, \]
is $\calD$-Markovian.

For every play $h \in H^\infty$ and every partition $\calD$ of the set of states,
denote by $Z(h;\calD) \in \{0,1,2,\cdots,\infty\}$ the number of times in which the play switches between elements of $D$ along $h$:
\[ Z(h;\calD) := \sup \{ k  \geq 0 \colon \tau^\calD_k < +\infty \}. \]

\subsection{The Partition According to the Uniform Max-Min Value of Player~$i$}

It will be useful to single out the partition $\calD^i_*$ of the set of states $S$ according to the uniform max-min value of player~$i$:
two states $s,s' \in S$ are in the same element of the partition $\calD^i_*$ if and only if $\underline v^i_1(s) = \underline v^i_1(s')$.
The following result implies that when player~$i$ plays an $\ep$-max-min preserving strategy,
the expected number of times that the play moves between elements of $\calD^i_*$ is bounded
by a constant that is independent of $\ep$ and the strategy profile of the other players.
Before stating the result we provide a weaker version of the concept of an $\ep$-max-min preserving strategy.

\begin{definition}
Let $i \in I$ be a player, let $\ep > 0$,
and let $\calD$ be a partition of the set of states.
A strategy $\sigma^i \in \Sigma^i$ is \emph{$(\calD,\ep)$-max-min preserving} if
for every initial state $s_0 \in S$,
every strategy profile of the other players $\sigma^{-i} \in \Sigma^{-i}$,
every play $h \in H^\infty$,
every $k \geq 0$,
and every bounded stopping time $\tau^{\calD}_k < \tau' \leq \tau^{\calD}_{k+1}$
we have
\begin{equation*}
%\label{equ:901}
\E_{s_{\tau^{\calD^i_*}_k},\sigma^i_{h_{\tau^{\calD^i_*}_k}},\sigma^{-i}_{h_{\tau^{\calD^i_*}_k}}}[\underline v^i_1(s_{\tau'})] \geq
\underline v^i_1(s_{\tau^{\calD^i_*}_k}) -  \ep.
\end{equation*}
\end{definition}
Note that every $\ep$-max-min preserving strategy is in particular $(\calD,\ep)$-max-min preserving, for every partition $\calD$ of the set of states.
Moreover, if the strategy $\sigma^i$ is $\ep$-max-min preserving, then for every partition $\calD$ the strategy $\sigma^i_{\calD}$
is $(\calD,\ep)$-max-min preserving.

Denote the minimal distance between distinct uniform max-min values of some player by
\[ \rho := \min\{ |\underline v^i_1(s)-\underline v^i_1(s')| \colon s,s' \in S, i \in I, \underline v^i_1(s) \neq \underline v^i_1(s')\}. \]

\begin{theorem}
\label{theorem:partition2}
For every $\delta > 0$
there is $\ep_0 > 0$ such that for every $\ep < \ep_0$,
every player $i \in I$,
every $(\calD^i_*,\ep)$-max-min preserving strategy $\sigma^i \in \Sigma^i$,
and every strategy profile $\sigma^{-i} \in \Sigma^{-i}$ there is an event $E$ for which $\prob_{s_0,\sigma^i,\sigma^{-i}}(E) > 1-\delta$ and
\begin{equation}
\label{equ:312}
\E_{s_0,\sigma^i,\sigma^{-i}}[Z(\cdot;\calD^i_*) \cdot \textbf{1}_E] \leq C(\rho),
\end{equation}
where $C(\rho)$ is some constant that is independent of $\delta$.
\end{theorem}

\begin{proof}
We first recall a general result on submartingales.
Let $\eta > 0$ and let $(W_n)_{n \geq 0}$ be a submartingale that satisfies the following properties for every $n \geq 0$:
\begin{itemize}
\item  The random variable $W_n$ attains values in the interval $[-1,2]$.
\item   Either $W_{n+1} = W_n$ or $|W_{n+1}-W_n| \geq \eta$.
\end{itemize}
Denote by $Z^*$ the number of times in which the process $(W_n)$ changes its value:
\begin{equation}
\label{equ:311}
Z^* := \#\{n \geq 0 \colon W_n \neq W_{n+1}\}.
\end{equation}
Then there is a constant $C(\eta)$ such that $\E[Z^*] \leq C(\eta)$.
Moreover, the function $\eta \mapsto C(\eta)$ can be taken to be monotonic nonincreasing.
This result can be deduced by the bound on the expected number of downcrossings of a bounded submartingale, see, e.g., Billingsley, 1995, Theorem 35.4.

Let now $\delta > 0$ and
fix two real numbers $K \geq \frac{C(\tfrac{\rho}{2})}{\delta}$ and $\ep_0 < \frac{\rho}{2K}$.
Define a stochastic process $(W_n)_{n=0}^\infty$ by
\[ W_n := \left\{
\begin{array}{lll}
\underline{v}^i_1(s_n) + \ep_0 k(\calD^i_*;n), & \ \ \ \ \ & W_{n-1} < 1+\tfrac{\rho}{2},\\
W_{n-1} & & W_{n-1} \geq 1+\tfrac{\rho}{2}.
\end{array}
\right.
\]
Note that $W_n < 1+\rho$ for every $n \geq 0$.
Ignoring the upper bound of $1+\tfrac{\rho}{2}$ that we impose on the process $(W_n)_{n \in \dN}$,
this process is equal to the uniform max-min value of the current state,
plus $\ep$ multiplied by the number of times in which the uniform max-min value changed along the play.

Let $\ep < \ep_0$, let $\sigma^i$ be a $(\calD^i_*,\ep)$-max-min preserving strategy of player~$i$,
and let $\sigma^{-i}$ be any strategy profile of the other players.
Since $\sigma^i$ is a $(\calD^i_*,\ep)$-max-min preserving strategy,
the process $(W_n)_{n=0}^\infty$ is a submartingale under $\prob_{s_0,\sigma^i,\sigma^{-i}}$.
Whenever the uniform max-min value of player~$i$ changes, the value of $W_n$ changes by at least $\rho - \ep_0$.
By Eq.~(\ref{equ:311}),
$\E_{s_0,\sigma^i,\sigma^{-i}}[Z^*] \leq C(\rho- \ep_0) \leq C(\tfrac{\rho}{2})$.
By Markov's Inequality,
$\prob_{s_0,\sigma^i,\sigma^{-i}}(Z^* < K) \geq 1-\tfrac{C(\tfrac{\rho}{2})}{K} \geq 1-\delta$.
Define the event
$E := \{ Z^* < K \}$.
On this event we have $W_n \leq 1 + K\ep < 1+\tfrac{\rho}{2}$ for every $n \geq 0$, and therefore on this event $Z(\cdot ; \calD^i_*) = Z^*$.
Moreover,
$\prob_{s_0,\sigma^i,\sigma^{-i}}(E) \geq 1-\delta$,
and
\[ \E_{s_0,\sigma^i,\sigma^{-i}}[Z(\cdot;\calD^i_*) \cdot \mathbf{1}_E] \leq \E_{s_0,\sigma^i,\sigma^{-i}}[Z^*] \leq C(\tfrac{\rho}{2}), \]
and the result follows.
\end{proof}

\bigskip

As a conclusion of Theorem~\ref{theorem:partition2} we deduce that in the notations of the statement of the theorem,
if the strategy $\sigma^i$ is in addition $\calD^i_*$-Markovian, then the expected number of times the uniform max-min value
of player~$i$ changes along the play is bounded.

\begin{corollary}
\label{cor:markovian}
In the notations of Theorem~\ref{theorem:partition2},
if the strategy $\sigma^i$ is in addition $\calD^i_*$-Markovian then
$\E_{s_0,\sigma^i,\sigma^{-i}}[Z(\cdot;\calD^i_*)] \leq \tfrac{4C(\rho)}{1-\delta}$.
\end{corollary}

\begin{proof}
By Eq.~(\ref{equ:312}) and Markov's inequality,
$\prob_{s_0,\sigma^i,\sigma^{-i}}(Z(\cdot;\calD^i_* \cdot \textbf{1}_E) \leq 2C(\rho)) \geq \tfrac{1}{2}(1-\delta)$.
Denoting $x:= \E_{s_0,\sigma^i,\sigma^{-i}}[Z(\cdot;\calD^i_*)]$ we deduce that
$x \leq 2C(\rho)+\tfrac{1}{2}(1+\delta)x$, and the claim follows.
\end{proof}

\bigskip

The following result%
\footnote{Abraham Neyman advised the author that Jean-Fran\c{c}ois Mertens and himself were aware of this result.}
implies in particular that every player~$i$ has a $\calD^i_*$-Markovian uniform $\ep$-max-min strategy, for every $\ep > 0$.

\begin{theorem}
\label{theorem:weakly:preserving}
For every $\delta > 0$ there is $\ep_0 > 0$ such that
for every $\ep < \ep_0$ and every uniform $\ep$-max-min strategy $\sigma^i$ that is $\ep$-max-min preserving,
the strategy $\sigma^i_{\calD^i_*}$ is uniform $\delta$-max-min and
$(\calD^i_*,\ep)$-max-min preserving.
\end{theorem}

\begin{proof}
Fix $\delta > 0$ and a player $i \in I$.
Let $\ep$ be smaller than the quantity $\ep_0$ given in Theorem~\ref{theorem:partition2} and smaller than $\tfrac{\delta\rho}{2C(\rho)}$,
where $C(\rho)$ is the constant given in Theorem~\ref{theorem:partition2}.
Let $\sigma^i$ be a uniform $\ep$-max-min strategy that is $\ep$-max-min preserving
and fix a strategy profile of the other players $\sigma^{-i} \in \Sigma^{-i}$.

By construction, the strategy $\sigma^i_{\calD^i_*}$ is $(\calD^i_*,\ep)$-max-min preserving.
It remains to show that the strategy $\sigma^i_{\calD^i_*}$ is a uniform $\delta$-max-min strategy.

For every two stopping times $\tau$ and $\tau'$ denote their minimum by
\[ \tau \wedge \tau' := \min\{\tau,\tau'\}. \]
Since the strategy $\sigma^i_{\calD^i_*}$ is $(\calD^i_*,\ep)$-max-min preserving,
for every $k \geq 0$ we have
\begin{equation}
\label{equ:Z2}
v(s_{\tau^{\calD^i_*}_k \wedge \tau}) \leq \E_{s_0,\sigma^i_{\calD^i_*},\sigma^{-i}}[v(s_{\tau^{\calD^i_*}_{k+1} \wedge \tau})] - \ep
\ \hbox{ on the event } \{ \tau^{\calD^i_*}_k \wedge \tau < \tau^{\calD^i_*}_{k+1} \wedge \tau\}.
\end{equation}
Using iteratively Eq.~(\ref{equ:Z2}) over all $k \geq 0$ we deduce that for every bounded stopping time $\tau$
\begin{eqnarray*}
\underline v^i_1(s_0) &\leq& \E_{s_0,\sigma^i_{\calD^i_*},\sigma^{-i}}[\underline v^i_1(s_{\tau})] -
\ep\left( 1 + \sum_{k \geq 0} \prob_{s_0,\sigma^i_{\calD^i_*},\sigma^{-i}}(\tau^{\calD^i_*}_k \wedge \tau < \tau^{\calD^i_*}_{k+1} \wedge \tau)\right)\\
&\leq& \E_{s_0,\sigma^i_{\calD^i_*},\sigma^{-i}}[\underline v^i_1(s_{\tau})] -
\ep C_2(\rho),
\end{eqnarray*}
where $C_2(\rho) = \tfrac{4C(\rho)}{1-\delta}$ and the last inequality follows from Corollary~\ref{cor:markovian}.
By averaging this inequality over all constant stopping times $\tau \in \{0,1,\cdots,N\}$ we conclude that for every $N \geq 0$,
\begin{equation}
\label{equ:313}
\underline v^i_1(s_0) \leq \frac{1}{N} \sum_{n=0}^{N-1} \underline v^i_1(s_n) - \ep C_2(\rho).
\end{equation}

By Theorem~\ref{theorem:partition2}
there is an event $E$ for which $\prob_{s_0,\sigma^i_{\calD^i_*},\sigma^{-i}}(E) > 1-\delta$ and
\begin{equation}
\label{equ:Z1}
\E_{s_0,\sigma^i_{\calD^i_*},\sigma^{-i}}[Z(\cdot;\calD^i_*) \cdot \textbf{1}_E] \leq C(\rho).
\end{equation}
By Eq.~(\ref{equ:Z1}) and Markov's inequality we deduce that
\begin{eqnarray}
\label{equ:prob:f}
\prob_{s_0,\sigma^i_{\calD^i_*},\sigma^{-i}}
\left(Z(\cdot;\calD^i_*) < \tfrac{\rho}{2\ep} \mid E\right) &\geq& 1-\tfrac{2\ep}{\rho} \E_{s_0,\sigma^i_{\calD^i_*},\sigma^{-i}}[Z(\cdot;\calD^i_*) \mid E]\\
&\geq& 1-\frac{2\ep}{\rho}C(\rho) \geq 1-\delta,
\ \ \ \forall \sigma^{-i} \in \Sigma^{-i}.
\nonumber
\end{eqnarray}
Denote by $F$ the event
\[ F := E \cap \left\{ Z(\cdot;\calD^i_*) < \tfrac{\rho}{2\ep} \right\}. \]
Since $\prob_{s_0,\sigma^i_{\calD^i_*},\sigma^{-i}}(E) > 1-\delta$ and
by Eq.~(\ref{equ:prob:f}), $\prob_{s_0,\sigma^i_{\calD^i_*},\sigma^{-i}}(F) \geq 1-2\delta$.
Since the strategy $\sigma^i_{\calD^i_*}$ is $\calD^i_*$-Markovian and since the strategy $\sigma^i$ is $\ep$-max-min preserving,
for every $k \geq 0$ and every bounded stopping time $\tau'$ that satisfies $\tau^{\calD^i_*}_k < \tau' \leq \tau^{\calD^i_*}_{k+1}$ we have
\begin{equation}
\label{equ:10a}
\gamma^i(s_0,\sigma^i_{\calD^i_*},\sigma^{-i}; \tau^{\calD^i_*}_k,\tau') \geq \underline v^i_1(s_{\tau^{\calD^i_*}_k}) - \ep
\
\hbox{ on the event }
\left\{ \E_{s_0,\sigma^i_{\calD^i_*},\sigma^{-i}}[\tau' - \tau^{\calD^i_*}_k \mid \calH(\tau^{\calD^i_*}_k)] \geq N_0 \right\},
\end{equation}
where $N_0$ is the constant given in Theorem~\ref{theorem:mn}.

On the event $F$ the strategy $\sigma^i$ restarts at most $\tfrac{\rho}{2\ep}$ times,
and therefore on this event the expected number of stages that take part in $\calD^i_*$-runs shorter than $N_0$ is at most $\tfrac{\rho N_0}{2\ep}$.
Since $\prob_{s_0,\sigma^i_{\calD^i_*},\sigma^{-i}}(F) > 1-2\delta$, since payoffs are bounded by 1,
and by Eqs.~(\ref{equ:313}) and~(\ref{equ:10a})
we deduce that for every stopping time $\tau$ that satisfies
$\E_{s_0,\sigma^i_{\calD^i_*},\sigma^{-i}}[\tau] \geq \tfrac{\rho N_0}{2\ep^2}$ we have
\[ \gamma^i(s_0,\sigma^i_{\calD^i_*},\sigma^{-i};0,\tau) \geq \underline v^i_1(s_0) - 3\ep -4\delta, \]
and therefore for every $\lambda$ sufficiently close to 1,
\[ \gamma^i_\lambda(s_0,\sigma^i_{\calD^i_*},\sigma^{-i}) \geq \underline v^i_1(s_0) - 4\ep - 4\delta. \]
The desired result follows.
\end{proof}

\bigskip

A useful property of uniform $\ep$-max-min $\calD^i_*$-Markovian strategies is that if the expected length of the $k$'th run is high,
then the average expected payoff during the $k$'th run is high.
This observation is summarized in the following Lemma.

\begin{lemma}
\label{lemma:markovian:maxmin}
Let $i \in I$ be a player,
let $\ep > 0$, let $\sigma^i \in \Sigma^i$ be a uniform $\ep$-max-min $\calD^i_*$-Markovian strategy,
and let $N_0 \in \dN$ be the constant given in Definition~\ref{defin:mn:maxmin}.
Then for every initial state $s_0$, every strategy profile $\sigma^{-i} \in \Sigma^{-i}$ of the other players,
and every $k \geq 0$,
on the set $\{ \E_{s_0,\sigma^i,\sigma^{-i}}[\tau^{\calD^i_*}_{k+1} - \tau^{\calD^i*}_{k} \mid \calH(\tau^{\calD^i_*}_{k}) ]\geq N_0\}$ we have
\[ \gamma(s_0,\sigma^i,\sigma^{-i};\tau^{\calD^i_*}_{k},\tau^{\calD^i_*}_{k+1}) \geq \underline v^i_1(s_{\tau^{\calD^i_*}_{k}})-\ep. \]
\end{lemma}

\section{The Modified Stochastic Game: The Max-Min Value}
\label{section:maxmin}

In this section we compare the limit as the discount factor goes to 1 of the max-min value in the modified game
to the uniform max-min value in the original stochastic game.

The max-min value of player~$i$ in the modified game $\widehat\Gamma_\lambda(s_0;\vec\calD,\vec c)$ is
\[ \underline{\widehat v}^i_{\lambda}(s_0;\calD^i,c^i) :=
\max_{\sigma^i \in \Sigma^i} \min_{\sigma^{-i} \in \Sigma^{-i}} \widehat \gamma_\lambda^i(s_0,\sigma;\calD^i,c^i). \]
Since the payoff function $\sigma \mapsto \widehat \gamma_\lambda^i(s_0,\sigma;\calD^i,c^i)$ is continuous (Lemma~\ref{lemma:continuous}),
the max-min value exists.
By Eq.~(\ref{equ:inequality:1}) we have $\widehat\gamma_\lambda^i(s_0,\sigma;\calD^i,c^i) \leq \gamma_\lambda^i(s_0,\sigma)$
for every strategy profile $\sigma \in \Sigma$,
and consequently
\begin{equation}
\label{equ:maxmin:comparison}
\underline{\widehat v}^i_{\lambda}(s_0;\calD^i,c^i) \leq {\underline v}^i_{\lambda}(s_0).
\end{equation}

We do not know whether the function $\lambda \mapsto \underline{\widehat v}^i_\lambda(s_0;\calD^i,c^i)$ is semi-algebraic,
and in particular we do not know whether the limit $\lim_{\lambda \to 1} \underline{\widehat v}^i_{\lambda}(s_0;\calD^i,c^i)$ always exists.
If the limit exists, we denote it by
$\underline {\widehat v}^i_1(s_0;\calD^i,c^i)$
and term it the \emph{limit max-min value} of player~$i$ in the modified game.

As the following example shows, the inequality in Eq.~(\ref{equ:maxmin:comparison}) can be strict,
even if the cutoff vector $c^i$ is high, and the difference between
${\underline v}^i_{\lambda}(s_0)$ and $\underline{\widehat v}^i_{\lambda}(s_0;\calD^i,c^i)$ need not vanish as $\lambda$ goes to 1.

\begin{example}
\label{example:2}
Consider the game that is depicted in Figure \arabic{figurecounter}.
In this game there is a single player, two states $S = \{s^0,s^1\}$, and the player has a single action in each state.
Set $\calD = \bigl\{ \{s^0\}, \{s^1\} \bigr\}$ and $c(D) = 4$ for every $D \in \calD$.

\bigskip

\centerline{
\begin{picture}(40,40)(0,-20)
\put( 0,0){\numbercellongb{$0$}{$^{s^1}$}}
\put( 20,-20){$s^0$}
\end{picture}
\ \ \
\ \ \
\begin{picture}(40,40)(0,-20)
\put( 0,0){\numbercellongb{$6$}{$^{s^0}$}}
\put( 20,-20){$s^1$}
\end{picture}
}

\centerline{Figure \arabic{figurecounter}: The game in Example~\ref{example:2}.}
\addtocounter{figurecounter}{1}
\bigskip

In this game the payoff alternates between 0 and 6, so that for each initial state
the $\lambda$-discounted value of this Markov chain converges to $3$ as
the discount factor $\lambda$ goes to 1.
In particular, the cutoff vector is higher than the $\lambda$-discounted value of the game for every discount factor $\lambda$ sufficiently close to 1.
In the modified game the payoff in state $s^1$ is 4,
so that the value of the modified game converges to $2$ as $\lambda$ goes to 1.
\end{example}

We will now provide conditions that ensure that the limit max-min value of a player in the modified game exists and coincides
with his uniform max-min value in the original game.
To this end we present a property, called \emph{Property P}, that some partitions satisfy.

\subsection{Property P}
\label{section:property:p}

We turn to formally present Property P.
Roughly, a partition $\calD^i$ satisfies Property P w.r.t.~player~$i$
if for every $\ep > 0$ there is a uniform $\ep$-max-min $\calD^i$-Markovian strategy of player~$i$
such that the number of times in which the play switches between elements of $\calD^i$ is uniformly bounded,
over the other players' strategy profile and over $\ep$.

\begin{definition}
\label{def:property:p}
Let $i \in I$ be a player.
A partition $\calD$ of the set of states $S$ \emph{satisfies Property~P w.r.t. player~$i$}
if there is a real number $\zeta > 0$ and for every $\ep > 0$ there is a uniform $\ep$-max-min $\calD$-Markovian strategy $\sigma^i \in \Sigma^i$ such that
for every strategy profile $\sigma^{-i} \in \Sigma^{-i}$ of the other players
there is an event $E$ for which $\prob_{s_0,\sigma^i,\sigma^{-i}}(E) > 1-\ep$ and
\[ \E_{s_0,\sigma^i,\sigma^{-i}}[Z(\cdot;\calD) \cdot \textbf{1}_E] \leq \zeta. \]
\end{definition}

\begin{example}[Absorbing Games]
\label{example:property:p:1}
When the game is an absorbing game,
the state variable can change at most once,
hence the element of the partition that contains the current state changes at most once.
It follows that any partition $\calD$ satisfies Property P w.r.t.~all players, with $\zeta=1$.
\end{example}

\begin{example}
\label{example:property:p:2}
Similarly to Example~\ref{example:property:p:1},
suppose that there is a set $D \subseteq S$ such that
(a) all states that are not in $D$ are absorbing, and
(b) $\calD = \{D, \{s\}_{s \not\in D}\}$.
Then the partition $\calD$ satisfies Property~P w.r.t.~all players, with $\zeta=1$.
In the application in Section~\ref{section:application} we will use this partition.
\end{example}

\begin{example}[The partition $\calD^i_*$]
\label{example:property:p:3}
The partition $\calD^i_*$ of the set of states according to the uniform max-min value of player~$i$
satisfies Property P w.r.t.~player~$i$.
Indeed, by Theorems~\ref{theorem:mn} and~\ref{theorem:weakly:preserving}, for every $\ep > 0$ player~$i$ has a $\calD^i_*$-Markovian
uniform $\ep$-max-min strategy $\sigma^i$,
and therefore by Theorem~\ref{theorem:partition2} this strategy satisfies Property P.
\end{example}

\subsection{A Result in Probability}

We would like to prove that when the partition $\calD^i$ satisfies Property P, and the cutoffs $(c^i(D))_{D \in \calD^i}$
are sufficiently high, then the limit max-min value of a player in the modified game exists and coincides with his uniform max-min value in the original game.
To this end we will show that a uniform $\ep$-max-min $\calD^i$-Markovian strategy $\sigma^i$ of player~$i$ guarantees to him in the modified game at least
$\underline v^i_1(s_0) - 2\ep$.
By Lemma~\ref{lemma:markovian:maxmin},
and because the strategy $\sigma^i$ is uniform $\ep$-max-min $\calD^i$-Markovian,
in every visit to an element of $\calD^i$ whose expected length is high, the expected average payoff is at least $\underline v^i_1(s_0) - \ep$.
Property P assures that the expected number of runs is bounded, hence the expected number of stages that belong to short visits to
elements of $\calD^i$ is uniformly bounded.
In particular, most stages belong to long visits to elements of $\calD^i$.
To show that the payoff in the modified game is high, we need to show that for every element $D$ of $\calD^i$,
the discounted payoff during long visits to $D$ is high.
In this section we provide a result that will help us deduce that when the expected average payoff during long visits to $D$ is high,
so is the expected discounted payoff.

Let $(\Omega,\calF,P)$ be a probability space, let $(\calF_n)_{n \geq 0}$ be a filtration,
and let $(\tau_k)_{k \geq 0}$ be a sequence of stopping times adapted to the filtration $(\calF_n)_{n \geq 0}$ such that $\tau_0= 0$
and $\tau_{k+1} > \tau_k$ on the event $\{\tau_k < \infty\}$, for every $k \geq 0$.
Denote $k_* := \sup_{k \geq 0} \{ k \geq 0 \colon \tau_k < \infty\}$.

For every nonnegative integer $l$ denote by $k(l)$ the unique nonnegative integer that satisfies $\tau_{k(l)} \leq l < \tau_{k(l)+1}$,
and denote $\tau_k^l := \tau_k \wedge l = \min\{\tau_k,l\}$.

\begin{theorem}
\label{theorem:stochastic}
In the notations above,
If there is $\zeta > 0$ such that $\E[k_*] \leq \zeta$,
then for every $\ep > 0$ and every $N \geq 0$ there are $M  = M(\ep,N,\zeta) \geq 0$ and $L_0 = L_0(\ep,N,\zeta) \geq 0$
(both are independent of the probability measure $P$)
that satisfy
\[ \sum_{l=0}^L
P\bigl(
\{ \E[(\tau_{k(l)+1}^L-\tau_{k(l)}^L)\wedge M \mid \calF_{\tau_{k(l)}}] \geq N\}\bigr) \geq (1-\ep)(L+1), \ \ \ \forall L \geq L_0. \]
\end{theorem}

For every $k \geq 0$ call the sequence of stages between stages $\tau_k$ and $\tau_{k+1}-1$ a \emph{$\tau$-run}.
The condition that $\E[k_*]$ is finite assures that
on average there are few $\tau$-runs up to stage $L$, hence,
when we restrict attention to stages $\{0,1,2,\cdots,L\}$,
the expected length of $\tau$-runs is large.
Theorem~\ref{theorem:stochastic} states a stronger property: there is $M$ such that
the expected length of most $\tau$-runs, restricted to their first $M$ stages, is high.
Moreover, the constant $M$ and the finite horizon $L_0$ are independent of the probability measure $P$.
In our application, the probability measure $P$ will be determined by the strategies of the players.
Since we study the max-min value, we need the quantities $M$ and $L_0$ to be independent
of the strategies of the other player, hence $M$ and $L_0$ need to be independent of $P$.

The next example shows that the quantity $M$ in Theorem~\ref{theorem:stochastic} can be large when $\zeta$ is large.
\begin{example}
Let $p \in (0,1)$, and consider a coin with parameter $P(\hbox{head}) = p$ that is drawn over and over.
For every $k \geq 0$ let $\tau_k$ be the following stopping time:
\[ \tau_k := \left\{
\begin{array}{lll}
k, & \ \ \ \ \ & \hbox{if the first } k \hbox{ draws are head},\\
\infty, & &  \hbox{otherwise.}
\end{array}
\right. \]
The first draw in which the outcome of the draw is tail is $k_*$ (recall that time starts at 0),
hence $P(k_*=k) = (1-p)p^{k}$, and in particular $\E[k_*] = \tfrac{1}{1-p}$.
Additionally, on the event $\{\tau_k < \infty\}$ we have $\E[(\tau_{k+1}-\tau_k) \wedge \tfrac{N-p}{1-p} \mid \tau_k < \infty] = N$.
It follows that
the constant $M$ in Theorem~\ref{theorem:stochastic} may be at least $\tfrac{N-p}{1-p}$.
\end{example}

\bigskip

\begin{proof}[Proof of Theorem~\ref{theorem:stochastic}]
We will substitute $M = \lceil\tfrac{2\zeta N}{\ep}\rceil$ and $L_0 = \lceil\tfrac{2\zeta M}{\ep}\rceil$,
where $\lceil x \rceil$ is the least integer larger than or equal to the real number $x$.
Denote
\[ G_k := \{ \tau_k < \infty\} \cap \left\{ \tau_{k+1} - \tau_k < M \right\}. \]
Then
\begin{equation}
\label{equ:p}
\zeta \geq \E[k_*]
= \sum_{k=0}^\infty P(\tau_k < \infty)
\geq \sum_{k=0}^\infty P(G_k).
\end{equation}
Since $\tau^L_{k+1} - \tau^L_k \leq \tau_{k+1}-\tau_k$, and since on $G_k$ we have
$\tau_{k+1} - \tau_k < M$,
it follows that for every $L \geq 0$,
\begin{equation}
\label{equ:g1}
\E\left[ \sum_{k=0}^\infty (\tau^L_{k+1} - \tau^L_k) \mathbf{1}_{G_k} \right] < M \sum_{k=0}^\infty P(G_k) \leq \zeta M.
\end{equation}
Denote
\[ E_k := \{ \tau_k < \infty\} \cap \left\{ \E\left[(\tau_{k+1}-\tau_k)\wedge  M \mid \calF_{\tau_k}\right] < N\right\}. \]
By Markov's Inequality, on $E_k$ we have
\[ P(\tau_{k+1}-\tau_k <  M \mid \calF_{\tau_k}) = P((\tau_{k+1}-\tau_k)\wedge M < M \mid \calF_{\tau_k})
> 1 - \tfrac{N}{M} \geq 1-\tfrac{\ep}{2\zeta}. \]
This implies that
\[ P(E_k \cap G_k) \geq (1-\tfrac{\ep}{2\zeta})P(E_k), \]
which subsequently implies that
\begin{equation}
\label{equ:26.1}
P(E_k \setminus G_k) \leq \tfrac{\ep}{2\zeta} P(E_k).
\end{equation}
As in Eq.~(\ref{equ:p}) we have $\sum_{k=0}^\infty P(E_k) \leq \zeta$,
and therefore by Eq.~(\ref{equ:26.1})
\[ \sum_{k=0}^\infty P(E_k \setminus G_k) \leq \tfrac{\ep}{2\zeta}\sum_{k=0}^\infty P(E_k) \leq \tfrac{\ep}{2}. \]
Since $\sum_{k=0}^\infty (\tau^L_{k+1} - \tau^L_k) = L+1$,
\begin{equation}
\label{equ:d1}
\E\left[ \sum_{k=0}^\infty (\tau^L_{k+1} - \tau^L_k) \mathbf{1}_{E_k \setminus G_k} \right] \leq \tfrac{\ep}{2} (L+1).
\end{equation}
From Eqs.~(\ref{equ:g1}) and~(\ref{equ:d1}) we obtain that
\begin{eqnarray*}
L+1 &=& \E\left[ \sum_{k=0}^\infty (\tau^L_{k+1} - \tau^L_k) \right] \\
&=&
\E\left[ \sum_{k=0}^\infty (\tau^L_{k+1} - \tau^L_k) \mathbf{1}_{(E_k)^c}\right] +
\E\left[ \sum_{k=0}^\infty (\tau^L_{k+1} - \tau^L_k) \mathbf{1}_{E_k \cap G_k}\right] +
\E\left[ \sum_{k=0}^\infty (\tau^L_{k+1} - \tau^L_k) \mathbf{1}_{E_k \setminus G_k}\right]\\
&\leq&
\E\left[ \sum_{k=0}^\infty (\tau^L_{k+1} - \tau^L_k) \mathbf{1}_{(E_k)^c}\right] +
\E\left[ \sum_{k=0}^\infty (\tau^L_{k+1} - \tau^L_k) \mathbf{1}_{G_k}\right] +
\E\left[ \sum_{k=0}^\infty (\tau^L_{k+1} - \tau^L_k) \mathbf{1}_{E_k \setminus G_k}\right]\\
&\leq&
\E\left[ \sum_{k=0}^\infty (\tau^L_{k+1} - \tau^L_k) \mathbf{1}_{(E_k)^c}\right] +
\zeta M + \tfrac{\ep}{2} (L+1)\\
&\leq&
\E\left[ \sum_{k=0}^\infty (\tau^L_{k+1} - \tau^L_k) \mathbf{1}_{(E_k)^c}\right] + \ep (L+1).
\end{eqnarray*}
The result follows since
\[ \E\left[ \sum_{k=0}^\infty (\tau^L_{k+1} - \tau^L_k)  \mathbf{1}_{(E_k)^c} \right]
=
\sum_{l=0}^L
P\bigl(
\{ \E[(\tau_{k(l)+1}^L-\tau_{k(l)}^L)\wedge M \mid \calF_{\tau_{k(l)}}] \geq N\}\bigr).
\]
\end{proof}

\subsection{Bounding the Max-Min Value in the Modified Game}

Our goal in this section is to show that when the partition $\calD^i$ satisfies Property P w.r.t.~player~$i$
and the cutoff $c^i(D)$ is at least the uniform max-min value in all states in $D$, for every element $D \in \calD^i$,
then the limit max-min value of player~$i$ in the modified game
exists and is equal to the uniform max-min value of the initial state.

\begin{theorem}
\label{theorem:5}
Let $\calD^i$ be a partition of the set of states $S$ that satisfies Property P w.r.t.~player~$i$
and let $c^i \in \dR^{\calD^i}$ satisfy $c^i(D) \geq \underline v_1^i(s)$ for every element $D \in \calD^i$ and every state $s \in D$.
Then for every initial state $s_0 \in S$
the limit $\widehat {\underline v}^i_1(s_0;\calD^i,c^i) := \lim_{\lambda \to 1} \underline{\widehat v}^i_{\lambda}(s_0;\calD^i,c^i)$ exists
we have $\widehat {\underline v}^i_1(s_0;\calD^i,c^i) = \underline v^i_1(s_0)$.
\end{theorem}

To prove Theorem~\ref{theorem:5} we will show that any uniform $\ep$-max-min $\calD^i$-Markovian strategy $\sigma^i_{\calD,\ep}$ guarantees
approximately $\underline v^i_1(s_0)$
in the modified game $\widehat \Gamma_\lambda(s_0;\calD,c)$, for every discount factor $\lambda$ sufficiently close to 1.

For the proof of the theorem we will need two notations and an observation.
For every discount factor $\lambda \in [0,1)$, every strategy profile $\sigma$, and every two bounded stopping times $0 \leq \tau \leq \tau'$
denote the expected unnormalized discounted time between the stopping times $\tau$ and $\tau'$ by
\[ t_\lambda(s_0,\sigma; \tau,\tau') :=
\E_{s_0,\sigma}\left[\sum_{n=\tau}^{\tau'-1} \lambda^{n-\tau} \mid \calH(\tau)\right]. \]
Denote the expected normalized discounted payoff between the stopping times $\tau$ and $\tau'$ by
\begin{equation}
\label{equ:180}
\gamma^i_\lambda(s_0,\sigma; \tau,\tau')
:= \frac{\E_{s_0,\sigma}\left[\sum_{n=\tau}^{\tau'-1} \lambda^{n-\tau} u^i(s_n,a_n) \mid \calH(\tau)\right]}
{t_\lambda(s_0,\sigma; \tau,\tau')}.
\end{equation}
The quantities $t_\lambda(s_0,\sigma; \tau,\tau')$ and
$\gamma^i_\lambda(s_0,\sigma; \tau,\tau')$
are random variables that depend on the information at stage $\tau$.

As is well known, the discounted payoff can be presented as a convex combination of the average payoffs.
We will use the following relation,
which relates the partial discounted sum to arithmetic sums,
and holds for every $\lambda \in [0,1)$, every $L \in \dN$, every sequence of real numbers $(x_n)_{n=0}^L$, and every $M \leq L$:
\begin{eqnarray}
\label{equ:501}
\sum_{n=0}^L\lambda^n x_n
= \sum_{n=0}^{M-1} (\lambda^n-\lambda^M)x_n +
\sum_{l=M}^{\infty} \left(((L \wedge l)+1)(\lambda^l - \lambda^{l+1}) \frac{\sum_{n=0}^{L \wedge l} x_n}{(L \wedge l)+1}\right).
\end{eqnarray}
Moreover,
when $M$ is fixed, the ratio
$\frac{\sum_{n=0}^{M-1} (\lambda^n-\lambda^M)}{\sum_{n=0}^L\lambda^n}$
of the sum of the coefficients in the first term on the right-hand side of Eq.~(\ref{equ:501})
and the sum of the coefficients on the left-hand side
goes to 0 as $\lambda$ goes to 1.
When replacing the nonnegative integer $L$ by a stopping time $\tau$, Eq.~(\ref{equ:501}) translates to
\begin{eqnarray}
\label{equ:501a}
\E\left[\sum_{n=0}^\tau\lambda^n x_n\right]
= \E\left[\sum_{n=0}^{(\tau \wedge M)-1} (\lambda^n-\lambda^{\tau\wedge M})x_n\right] +
\sum_{l=M}^\infty (\lambda^l - \lambda^{l+1})\E[((\tau\wedge l)+1)] \frac{\E[\sum_{n=0}^{\tau\wedge l} x_n]}{\E[(\tau\wedge l)+1]}.
\end{eqnarray}
Moreover, the ratio
$\frac{\E\left[\sum_{n=0}^{(\tau\wedge M)-1} (\lambda^n-\lambda^{\tau\wedge M})\right]}{\E\left[\sum_{n=0}^\tau\lambda^n\right]}$
goes to 0 as $\lambda$ goes to 1.

\bigskip

\begin{proof}
In view of Eq.~(\ref{equ:maxmin:comparison}) we need to show that $\widehat {\underline v}^i_1(s_0;\calD^i,c^i) \geq \underline v^i_1(s_0)$.
Fix $\delta > 0$.
Let $\zeta$ be the constant of Definition~\ref{def:property:p},
let $\ep > 0$ be the constant given by Theorem~\ref{theorem:weakly:preserving},
let $\sigma^i$ be the uniform $\ep$-max-min $\calD^i$-Markovian strategy for player~$i$ given by Property P,
and let $\sigma^{-i}$ be any strategy profile of the other players.
We will show that the strategy $\sigma^i$ guarantees that
player~$i$'s payoff in the modified game $\widehat\Gamma_\lambda(s_0;\calD^i,c^i)$
is at least $\underline v^i_1(s_0) - 5\ep-2\delta$,
provided the discount factor is sufficiently close to 1.

Let $N_0$ be the constant given by Theorem~\ref{theorem:mn},
and let $M$ and $L_0$ be the constant given by Theorem~\ref{theorem:stochastic} for $\zeta$
and $N=N_0$.
Let $\lambda$ be sufficiently close to 1 such that the ratio
$\frac{\E\left[\sum_{n=0}^{(\tau\wedge M)-1} (\lambda^n-\lambda^{\tau\wedge M})\right]}{\E\left[\sum_{n=0}^\tau\lambda^n\right]}$
is smaller than $\ep$.

We argue that on the event
$\left\{ \E_{s_0,\sigma^i_{\calD,\ep},\sigma^{-i}}[\tau^\calD_{k+1} - \tau^\calD_k \mid \calH(\tau^\calD_k)] \geq N_0 \right\}$
we have
\begin{equation}
\label{equ:30.1}
\gamma^i_\lambda(s_0,\sigma; \tau^\calD_k,\tau^\calD_{k+1}-1) \geq \underline v^i_1(s_{\tau^\calD_k}) - 3\ep.
\end{equation}
Indeed, on this event
\begin{eqnarray}
\nonumber
&&\gamma^i_\lambda(s_0,\sigma; \tau^{\calD^i}_k,\tau^{\calD^i}_{k+1}-1)\\
\label{equ:181}
&&=
\frac{\E_{s_0,\sigma}\left[\sum_{n=\tau^{\calD^i}_k}^{\tau^{\calD^i}_{k+1}-1} \lambda^{n-\tau^{\calD^i}_k} u^i(s_n,a_n) \mid \calH(\tau^{\calD^i}_{k})\right]}
{t_\lambda(s_0,\sigma; \tau^{\calD^i}_k,\tau^{\calD^i}_{k+1})}\\
\label{equ:182}
&&= \frac{\E\left[\sum_{n=0}^{(\tau^{\calD^i}_{k} \wedge M)-1} (\lambda^n-\lambda^{\tau\wedge M})u^i(s_n,a_n) \mid \calH(\tau^{\calD^i}_k)\right]}
{t_\lambda(s_0,\sigma; \tau^{\calD^i}_k,\tau^{\calD^i}_{k+1})}\\
\nonumber
&&\ \  +
\frac{\sum_{l=M}^\infty (\lambda^l - \lambda^{l+1})\E[((\tau^{\calD^i}_k\wedge l)+1)] \gamma^i(s_0,\sigma^i,\sigma^{-i}; \tau^{\calD^i}_{k},\tau^{\calD^i}_{k+1})}
{t_\lambda(s_0,\sigma; \tau^{\calD^i}_k,\tau^{\calD^i}_{k+1})}\\
\label{equ:183}
&&\geq
\sum_{l=M}^\infty (\lambda^l - \lambda^{l+1})\E[((\tau^{\calD^i}_k\wedge l)+1)] \gamma^i(s_0,\sigma^i,\sigma^{-i}; \tau^{\calD^i}_{k},\tau^{\calD^i}_{k+1})-\ep\\
\label{equ:184}
&&\geq
(\underline v^i_1(s_0) - \ep)\sum_{l=M}^\infty (\lambda^l - \lambda^{l+1})\E[((\tau^{\calD^i}_k\wedge l)+1)]  - \ep\\
\label{equ:185}
&&\geq
(\underline v^i_1(s_0) - \ep)(1-\ep)-\ep\\
\label{equ:186}
&&\geq \underline v^i_1(s_0) - 3\ep,
\end{eqnarray}
where Eq.~(\ref{equ:181}) holds by definition (see Eq.~(\ref{equ:180})),
Eq.~(\ref{equ:182}) holds by Eq.~(\ref{equ:501a}),
Eq.~(\ref{equ:183}) holds by the choice of $\lambda$,
Eq.~(\ref{equ:184}) holds by the choice of $M$,
and Eq.~(\ref{equ:185}) holds by the choice of $\lambda$.

Finally, the payoff in the modified game satisfies
\begin{eqnarray}
\nonumber
&&\widehat \gamma^i_\lambda(s_0,\sigma^i,\sigma^{-i};\calD^i,c^i)\\
&&\ \ =
\label{equ:582}
\E_{s_0,\sigma^i,\sigma^{-i}}
\left[ \sum_{D \in \calD^i}
\min\left\{U^i_\lambda(s_0,\sigma^i,\sigma^{-i};D), t_\lambda(s_0,\sigma^i,\sigma^{-i};D)\cdot  c^i(D)\right\} \right]\\
\label{equ:583}
&&\ \ \geq
\sum_{D \in \calD^i} \sum_{k=0}^\infty
\E_{s_0,\sigma^i,\sigma^{-i}}
\left[
\mathbf{1}_{\{\tau^{\calD^i}_k < \infty\} \cap \{D(s_{\tau^{\calD^i}_k}) = D\}} \cdot
\lambda^{\tau^{\calD^i}_k-1} \cdot\right.\\
\nonumber
&&\ \ \ \ \ \ \ \ \ \ \ \ \ \ \ \ \ \ \ \ \
\left.
t_\lambda(s_0,\sigma; \tau^{\calD^i}_k,\tau^{\calD^i}_{k+1})\min\{\gamma^i_\lambda(s_0,\sigma^i,\sigma^{-i}; \tau^{\calD^i}_k,\tau^{\calD^i}_{k+1}), c^i(D)\}
\right]
\\
\nonumber
&&\ \ \geq
\sum_{D \in \calD^i} \sum_{k=0}^\infty
\E_{s_0,\sigma^i,\sigma^{-i}}
\left[
\mathbf{1}_{\{\tau^{\calD^i}_k < \infty\} \cap \{D(s_{\tau^{\calD^i}_k}) = D\}}
t_\lambda(s_0,\sigma^i,\sigma^{-i}; \tau^{\calD^i}_k,\tau^{\calD^i}_{k+1}) \cdot (\underline v^i_1(s_{\tau^{\calD^i}_k})-3\ep)
\right]\\
&&\ \ \ \ \ \ \ \ \ \ \ \ \ \ \ \ \ \ \ \ \ -2\ep
\label{equ:584}
\\
&&\ \ \geq
\underline v^i_1(s_0)-5\ep-2\delta.
\label{equ:585}
\end{eqnarray}
where
Eq.~(\ref{equ:582}) holds by the definition of the modified payoff,
Eq.~(\ref{equ:583}) follows holds since
\begin{equation}
\label{equ:min}
\min\{a_1,b_1\} + \min\{a_2,b_2\} \leq \min\{a_1+a_2,b_1+b_2\}, \ \ \ \forall a_1,a_2,b_1,b_2\in \dR,
\end{equation}
Eq.~(\ref{equ:584}) holds by Theorem~\ref{theorem:stochastic}, Eq.~(\ref{equ:30.1}), and
since $c^i(D) \geq \underline v^i_1(s)$ for every element $D \in \calD^i$ and every state $s \in D$,
and Eq.~(\ref{equ:585}) holds since $\sigma^i$ is $\delta$-max-min preserving.
\end{proof}

\section{The Modified Game: the Min-Max Value}
\label{section:minmax}

In Section~\ref{section:maxmin} we studied the max-min value in the modified game.
In the present section we derive the analogous results with regards to the min-max value.
The proofs of these results are similar to the proofs regarding the max-min value, hence omitted.
The reason we started with the study of the max-min value is that while player~$i$ has one strategy that guarantees that his payoff
is at least his max-min value minus $\ep$, whatever the other players play,
this is not the case for the min-max value:
for every strategy profile of the other players, player~$i$ may have a different strategy that guarantees
that his payoff is at least the uniform min-max value minus $\ep$.
This feature of the concept of the min-max value has the implication that definitions and proofs are more cumbersome,
yet they pose no new technical difficulties.

The \emph{$\lambda$-discounted min-max value of player $i$ at the initial state $s_0$} in the original stochastic game is given by
\begin{equation}
\label{equ:minmax}
\overline v_\lambda^i(s_0) := \min_{\sigma^{-i} \in \Sigma^{-i}} \max_{\sigma^i \in \Sigma^i} \gamma_\lambda^i(s_0;\sigma^i,\sigma^{-i}).
\end{equation}
The interpretation of the min-max value is that the other players can ensure that player~$i$'s payoff will not be above his min-max value,
and they cannot lower his payoff further.
Denote
\[ \overline v_1^i(s_0) := \lim_{\lambda \to 1} \overline v_\lambda^i(s_0), \]
which is termed the \emph{uniform min-max value} of player~$i$ at the initial state $s_0$.

Neyman (2003) proved that for every strategy profile of players $I \setminus \{i\}$ player~$i$
has a response that guarantees that his payoff is at least the uniform min-max value minus $\ep$,
provided the discount factor is sufficiently close to 1 or the game is sufficiently long.
We will need this strategy to satisfy an additional condition,
which resembles the concept of being Markovian w.r.t.~a partition.
This condition is complicated to formulate, since the strategy profile of players $I \setminus \{i\}$ need not be Markovian w.r.t.~the partition.
We therefore spell out directly the properties that the response of player~$i$ should satisfy.

Denote by $\calD^i_{**}$ the partition of the set of states $S$ according to the uniform min-max value of player~$i$:
two states $s,s' \in S$ are in the same element of the partition if and only if $\overline v^i_1(s) = \overline v^i_1(s')$.
The next result follows from Neyman (2003) together with the analog of Theorem~\ref{theorem:weakly:preserving}.

\begin{theorem}
\label{theorem:mn1}
For every $\ep > 0$,
every player $i \in I$,
and every initial state $s_0 \in S$
there is $N_0 \in \dN$
such that for every strategy profile $\sigma^{-i} \in \Sigma^{-i}$
there is a strategy $\sigma^i$ such that
for every $N \geq N_0$,
\begin{eqnarray}
\gamma_N^i(s_0,\sigma^i,\sigma^{-i}) &\geq& \overline v^i_1(s_0)- \ep.
\label{equ:mn:15}
\end{eqnarray}
Moreover, for every $k \geq 0$ and every bounded stopping time $\tau \geq \tau^{\calD^i_{**}}_k$,
\begin{equation*}
%\label{equ:mn:13}
\overline v^i_1(s_{\tau^{\calD^i_{**}}_k}) \leq
\E_{s_{\tau^{\calD^i_{**}}_k},\sigma^i_{h_{\tau^{\calD^i_{**}}_k}},\sigma^{-i}_{h_{\tau^{\calD^i_{**}}_k}}}[\overline v^i_1(s_\tau)] + \ep,
\end{equation*}
and on the event $\left\{\E_{s_0,\sigma^i,\sigma^{-i}}[\tau - \tau^{\calD^i_{**}}_k \mid \calH(\tau^{\calD^i_{**}}_k)] \geq N_0\right\}$ we have
\begin{equation}
\label{equ:mn:12}
\gamma^i(s_0,\sigma^i,\sigma^{-i};\tau^{\calD^i_{**}}_k,\tau) \geq \overline v^i_1(s_{\tau^{\calD^i_{**}}_k})- \ep.
\end{equation}
\end{theorem}

Note that Eq.~(\ref{equ:mn:12}) requires some sort of uniform subgame perfectness:
the strategy profile $\sigma^i$ is a good response in every subgame that starts when the play switches an element of the partition $\calD^i_{**}$.
We will call a strategy $\sigma^i$ that satisfies Eqs.~(\ref{equ:mn:15})--(\ref{equ:mn:12}) for the partition $\calD^i_{**}$
a \emph{uniform $\calD^i_{**}$-subgame-perfect $\ep$-min-max strategy against $\sigma^{-i}$}.

The analog of Property P for the min-max value is the following.

\begin{definition}
\label{def:property:ptag}
Let $i \in I$ be a player.
A partition $\calD$ of the set of states $S$ \emph{satisfies Property~P' w.r.t.~player $i$}
if there is a real number $\zeta > 0$ and for every $\ep > 0$
and every strategy profile $\sigma^{-i} \in \Sigma^{-i}$ there
are a uniform $\calD$-subgame-perfect $\ep$-min-max strategy $\sigma^i \in \Sigma^i$ against $\sigma^{-i}$
and an event $E$ such that $\prob_{s_0,\sigma^i,\sigma^{-i}}(E) \geq 1-\ep$ and
\begin{eqnarray*}
%\label{equ:mn:23}
\E_{s_0,\sigma^i,\sigma^{-i}}[Z \cdot \mathbf{1}_E] &\leq& \zeta.
\end{eqnarray*}
\end{definition}

The partitions in Examples~\ref{example:property:p:1} and~\ref{example:property:p:2} satisfy Property~P' w.r.t.~all players.
Analogously to Example~\ref{example:property:p:3},
for every player $i \in I$ the partition~$\calD^i_{**}$ satisfies Property P' w.r.t. player~$i$.

Denote by
\[ \widehat {\overline v}^i_\lambda(s_0;\calD^i,c^i) :=
\min_{\sigma^{-i} \in \Sigma^{-i}} \max_{\sigma^i \in \Sigma^i} \gamma_\lambda^i(s_0;\sigma^i,\sigma^{-i}) \]
the min-max value of player~$i$ in the modified game $\widehat\Gamma^i_\lambda(s_0;\vec\calD,\vec c)$.
The following result, which is analogous to Theorem~\ref{theorem:5},
states that if the partition $\calD^i$ satisfies Property P' and if the vector $c^i \in \dR^{\calD^i}$ is at least the min-max value,
then the limit min-max value of player~$i$ in the modified game is his uniform min-max value at the initial state in the original stochastic game.
\begin{theorem}
\label{theorem:6}
Let $\calD^i$ be a partition of the set of states $S$ that satisfies Property P' w.r.t.~player~$i$
and let $c^i \in \dR^{\calD^i}$ satisfy $c^i(D) \geq \overline v^i_1(s)$ for every element $D \in \calD^i$ and every state $s \in D$.
Then for every initial state $s_0 \in S$
the limit $\widehat {\overline v}^i_1(s_0;\calD^i,c^i) := \lim_{\lambda \to 1} \widehat {\overline v}^i_\lambda(s_0;\calD^i,c^i)$ exists
and satisfies
$\widehat {\overline v}^i_1(s_0;\calD^i,c^i) = \overline v^i_1(s_0)$.
\end{theorem}

When particularized to absorbing games, Theorem~\ref{theorem:6} coincides with Step~2 in the proof of Theorem~4.5 in Solan (1999).

Since the payoff to a player in an equilibrium is always at least his min-max value,
Theorem~\ref{theorem:6} yields that the limit of stationary equilibrium payoffs of each player~$i$ in the modified game
is at least his uniform min-max value in the original stochastic game at the initial state.

\begin{corollary}
\label{corollary:equilibria:minmax}
Fix an initial state $s_0 \in S$,
a vector of partitions $\vec\calD = (\calD^i)_{i \in I}$ of the set of states such that for every player $i \in I$
the partition $\calD^i$ satisfies Property P' w.r.t. player~$i$,
and a vector $\vec c = (c^i(D))_{i \in I, D \in \calD^i} \in \dR^{I \times\calD}$
such that $c^i(D) \geq \overline v^i_1(s)$ for every player $i \in I$, every element $D \in \calD^i$, and every state $s \in D$.
Let $\lambda \mapsto x_\lambda$ be a function that assigns to every discount factor $\lambda \in [0,1)$
a stationary equilibrium in the modified game $\widehat \Gamma_\lambda(s_0;\vec \calD,\vec c)$.
Then $\liminf_{\lambda \to 1} \widehat\gamma^i_\lambda(s_0,x_\lambda;\calD^i,c^i) \geq \overline v^i_1(s_0)$.
\end{corollary}

\section{The Min-Max Value in Stationary Strategies}
\label{section:minmax:stationary}

In discounted stochastic games the min-max (resp.~max-min) value of a player when all players are restricted to stationary strategies
coincides with his min-max (resp. max-min) value without this restriction.
In this section (resp. Section~\ref{section:maxmin:stationary}) we study the min-max (resp. max-min)
value in the modified game when the players are restricted to stationary strategies,
and check whether the above mentioned phenomenon holds in the modified game.
The motivation for this study is that if one can prove that there exists a stationary strategy profile $x_\lambda$ that guarantees that the payoff of each player
in the modified game is at least
his min-max value (resp.~max-min value), whatever be the initial state,
and if the limit min-max value (resp.~max-min value) in stationary strategies in the modified game when the players are restricted to stationary strategies
coincides with his uniform min-max value (resp.~uniform max-min value) in the original stochastic game,
then we may be able to derive results on the original stochastic game using the sequence $(x_\lambda)_{\lambda \in [0,1)}$.

The min-max value of player~$i$ in the modified game
$\widehat\Gamma_\lambda(s_0;\vec \calD,\vec c)$ when the players are
restricted to stationary strategies is
\begin{equation}
\label{equ:minmax:stationary}
\widehat{\overline v}^i_{\lambda,\hbox{stat}}(s_0;\calD^i,c^i) :=
\min_{x^{-i} \in \Sigmamistatft} \max_{x^i \in \Sigmaistatft} \widehat \gamma_\lambda^i(s_0,x^i,x^{-i};\calD^i,c^i).
\end{equation}
Since the payoff function $\widehat\gamma_\lambda^i(s_0,\cdot;\calD^i,c^i) : \Sigma \to \dR$ is continuous, the min-max value
in stationary strategies is well-defined.
In Eq.~(\ref{equ:minmax:stationary}), the maximum and minimum are over the set of stationary strategies of the players, hence both
are smaller than the sets considered for the calculation of the min-max value $\overline v_\lambda^i(s_0;\calD^i,c^i)$.
Consequently, there is no clear comparison between the min-max value and the min-max value in
stationary strategies of the modified game.
Nevertheless we have the following result.

\begin{theorem}
\label{theorem:minmax:comparison}
For every discount factor $\lambda \in [0,1)$,
every initial state $s_0 \in S$,
every player $i\in I$,
every partition $\calD^i$ of the set of states,
and every vector $c^i \in \dR^{\calD^i}$,
\begin{equation}
\label{equ:100}
\widehat {\overline v}^i_\lambda(s_0;\calD^i,c^i) \leq
\widehat{\overline v}^i_{\lambda,\hbox{stat}}(s_0;\calD^i,c^i)
\leq \overline v^i_\lambda(s).
\end{equation}
\end{theorem}

\begin{proof}
The left-hand side inequality in Eq.~(\ref{equ:100}) holds due to the following chain of equalities and inequalities:
\begin{eqnarray}
\widehat{\overline v}^i_{\lambda,\hbox{stat}}(s_0;\calD^i,c^i) &=&
\min_{x^{-i} \in \Sigmamistatft} \max_{x^i \in \Sigmaistatft} \widehat \gamma_\lambda^i(s_0,x^i,x^{-i};\calD^i,c^i)\\
\label{equ:301}
&=& \min_{x^{-i} \in \Sigmamistatft} \max_{\sigma^i \in \Sigma^i} \widehat \gamma_\lambda^i(s_0,\sigma^i,x^{-i};\calD^i,c^i)\\
&\geq& \min_{\sigma^{-i} \in \Sigma^{-i}} \max_{\sigma^i \in \Sigma^i} \widehat \gamma_\lambda^i(s_0,\sigma^i,\sigma^{-i};\calD^i,c^i)\\
&=& \widehat {\overline v}^i_\lambda(s_0;\calD^i,c^i),
\end{eqnarray}
where Eq.~(\ref{equ:301}) follows from Theorem~\ref{proposition:stationary:mdp}.

We turn to prove the right-hand side inequality in Eq.~(\ref{equ:100}).
Because the discounted min-max value in the original stochastic game is attained in stationary strategies
and by Eq.~(\ref{equ:inequality:1}) we have
\begin{eqnarray*}
\overline v^i_\lambda(s_0)
&=& \min_{x^{-i} \in \Sigmamistatft} \max_{x^i \in \Sigmaistatft} \gamma^i_\lambda(s_0,x^i,x^{-i})\\
&\geq& \min_{x^{-i} \in \Sigmamistatft} \max_{x^i \in \Sigmaistatft} \widehat\gamma^i_\lambda(s_0,x^i,x^{-i};\calD^i,c^i)\\
&=& \widehat{\overline v}^i_{\lambda,\hbox{stat}}(s_0;\calD^i,c^i),
\end{eqnarray*}
as claimed.
\end{proof}

\bigskip

By taking the limits as $\lambda$ goes to 1 in Eq.~(\ref{equ:100})
and by using Theorem~\ref{theorem:6} we obtain the following result.
\begin{corollary}
\label{theorem:minmax:equality}
For every initial state $s_0 \in S$,
every player $i\in I$,
every partition $\calD^i$ of the set of states that satisfies Property P' w.r.t. player~$i$,
and every vector $c^i \in \dR^{\calD^i}$ that satisfies that $c^i(D) \geq \overline v^i_1(s)$ for every element $D \in \calD^i$ and every state $s \in D$,
we have
\begin{equation*}
%\label{equ:100a}
\lim_{\lambda \to 1} \widehat{\overline v}^i_{\lambda,\hbox{stat}}(s_0;\calD^i,c^i)
=
\lim_{\lambda \to 1} \widehat {\overline v}^i_\lambda(s_0;\calD^i,c^i).
\end{equation*}
\end{corollary}

\section{The Max-Min Value in Stationary Strategies}
\label{section:maxmin:stationary}

In this section we define the concept of max-min value in stationary strategies in the modified game,
and show that its behavior is different than the behavior of the min-max value in stationary strategies.

The max-min value of player~$i$ in the modified game $\widehat\Gamma_\lambda(s_0;\vec\calD,\vec c)$
when the players are restricted to stationary strategies is
\begin{equation*}
%\label{equ:maxmin:stationary}
\underline{\widehat v}^i_{\lambda,\hbox{stat}}(s_0;\calD^i,c^i) :=
\max_{x^i \in \Sigmaistatft} \min_{x^{-i} \in \Sigmamistatft} \widehat \gamma_\lambda^i(s_0,x^i,x^{-i};\calD^i,c^i).
\end{equation*}

We first show that the max-min value is not lowered when the other players are allowed to play any strategy profile.

\begin{theorem}
\label{theorem:maxmin:1}
For every discount factor $\lambda \in [0,1)$,
every initial state $s_0 \in S$,
every player $i\in I$,
every partition $\calD^i$ of the set of states,
and every vector $c^i \in \dR^{\calD^i}$,
\begin{equation*}
\underline{\widehat v}^i_{\lambda,\hbox{stat}}(s_0;\calD^i,c^i) =
\max_{x^i \in \Sigmaistatft} \min_{\sigma^{-i} \in \Sigma^{-i}} \widehat \gamma_\lambda^i(s_0,x^i,\sigma^{-i};\calD^i,c^i).
\end{equation*}
\end{theorem}

Theorem~\ref{theorem:maxmin:1} follows from the following result,
which states that when player~$i$ plays a stationary strategy,
the lowest payoff of player~$i$ in the modified game is obtained when
players $I \setminus \{i\}$ play a pure stationary strategy profile.

\begin{lemma}
\label{lemma:maxmin:1}
For every discount factor $\lambda \in [0,1)$,
every initial state $s_0 \in S$,
every player $i\in I$,
every partition $\calD^i$ of the set of states,
every vector $c^i \in \dR^{\calD^i}$,
every stationary strategy $x^i \in \Sigmaistat$,
and every strategy profile $\sigma^{-i} \in \Sigma^{-i}$ of the other players
there is a pure stationary strategy profile $x^{-i} \in \Sigmamistat$ such that
\[ {\widehat \gamma}_\lambda^i(s_0,x^i,\sigma^{-i};\calD^i,c^i) \geq {\widehat \gamma}_\lambda^i(s_0,x^i,x^{-i};\calD^i,c^i). \]
\end{lemma}

\begin{proof}
When player~$i$ plays a stationary strategy and the other players are conceived as a single decision maker who can correlate its $|I|-1$ actions,
the decision problem of players~$I \setminus \{i\}$ becomes a Markov decision problem.
By Altman (1999, Theorem 3.2)
there are $L \in \dN$,
a probability distribution $\beta \in \Delta(\{1,2,\ldots,L\})$,
and a collection of pure stationary strategy profiles $(x^{-i}_l)_{l=1}^L$ of all players other that player~$i$ such that
\[ t_\lambda(s_0,x^i,\sigma^{-i}) = \sum_{l=1}^L \beta_l t_\lambda(s_0,x^{i},x^{-i}_l). \]
By Eq.~(\ref{equ:105}), for every element $D \in \calD^i$,
\begin{eqnarray*}
U^i_\lambda(s_0,x^i,\sigma^{-i};D) &=&
\sum_{s \in D, a \in A(s)}t_\lambda(s_0,x^i,\sigma^{-i};s,a)u^i(s,a)\\
&=&
\sum_{s \in D, a \in A(s)}\sum_{l=1}^L \beta_l t_\lambda(s_0,x^i,x^{-i}_l)u^i(s,a)\\
&=& \sum_{l=1}^L \beta_l U^i_\lambda(s_0,x^i,x_l^{-i};D).
\end{eqnarray*}
By Eq.~(\ref{equ:min})
we have
\begin{eqnarray*}
\widehat\gamma_\lambda^i(s_0,x^i,\sigma^{-i};\calD^i,c^i)
&=&
\sum_{D \in \calD^i} \min\left\{U^i_\lambda(s_0,x^i,\sigma^{-i};D), t_\lambda(s_0,x^i,\sigma^{-i};D)\cdot  c^i(D)\right\}\\
&=&
\sum_{D \in \calD^i} \min\left\{\sum_{l=1}^L \beta_l U^i_\lambda(s_0,x^i,x_l^{-i};D), \sum_{l=1}^L \beta_l t_\lambda(s_0,x^i,x_l^{-i};D)\cdot  c^i(D)\right\}\\
&\geq&
\sum_{D \in \calD^i} \sum_{l=1}^L \beta_l \min\left\{U^i_\lambda(s_0,x^i,x_l^{-i};D), t_\lambda(s_0,x^i,x_l^{-i};D)\cdot  c^i(D)\right\}\\
&=&
\sum_{l=1}^L \beta_l \sum_{D \in \calD^i} \min\left\{U^i_\lambda(s_0,x^i,x_l^{-i};D), t_\lambda(s_0,x^i,x_l^{-i};D)\cdot  c^i(D)\right\}\\
&=&
\sum_{l=1}^L \beta_l \widehat\gamma_\lambda^i(s_0,x^i,x_l^{-i};\calD^i,c^i)\\
&\geq& \min_{l=1,\cdots,L} \widehat\gamma_\lambda^i(s_0,x^i,x_l^{-i};\calD^i,c^i),
\end{eqnarray*}
and the desired result follows.
\end{proof}

\bigskip

We can now prove that the max-min value in stationary strategies in the modified game is not larger than the max-min value in the modified game.
\begin{theorem}
\label{theorem:maxmin:comparison}
For every discount factor $\lambda \in [0,1)$,
every initial state $s_0 \in S$,
every player $i\in I$,
every partition $\calD^i$ of the set of states,
and every vector $c^i \in \dR^{\calD^i}$,
\begin{equation}
\label{equ:102}
\underline{\widehat v}^i_{\lambda,\hbox{stat}}(s_0;\calD^i,c^i)
\leq
\widehat {\underline v}^i_\lambda(s_0;\calD^i,c^i).
\end{equation}
\end{theorem}

\begin{proof}
The claim follows from the following chain of equalities and inequality, that holds due to Lemma~\ref{lemma:maxmin:1}.
\begin{eqnarray*}
\widehat{\underline{v}}^i_\lambda(s_0;\calD^i,c^i)
&=& \max_{\sigma^i \in \Sigma^i}\min_{\sigma^{-i} \in \Sigma^{-i}} \gamma_\lambda^i(s_0;\sigma^i,\sigma^{-i})\\
&\geq& \max_{x^i \in \Sigmaistatft}\min_{\sigma^{-i} \in \Sigma^{-i}} \gamma_\lambda^i(s_0;x^i,\sigma^{-i})\\
&=& \max_{x^i \in \Sigmaistatft}\min_{x^{-i} \in \Sigmamistatft} \gamma_\lambda^i(s_0;x^i,\sigma^{-i})\\
&=& \widehat{\underline{v}}^i_{\lambda,\hbox{stat}}(s_0;\calD^i,c^i).
\end{eqnarray*}
\end{proof}

\bigskip

The following example shows that the inequality in Theorem~\ref{theorem:maxmin:comparison} can be strict
and not vanish as the discount factor goes to 1.

\begin{example}[The Big Match]
\label{example:big match}
The two-player zero-sum absorbing game that appears in Figure \arabic{figurecounter}
is a slight variation of the Big Match,
which was introduced by Gillette (1956) and analyzed by Blackwell and Ferguson (1968).

\centerline{
\begin{picture}(130,85)(-10,-20)
\put(-10, 8){$B$}
\put(-10,28){$T$}
\put(30,50){$L$}
\put(90,50){$R$}
\put(60,-20){$s^0$}
\put(60,20){\numbercellongg{$1,0$}{$^{s^1}$}}
\put(0,20){\numbercellongg{$0,1$}{$^{s^2}$}}
\put(60,0){\numbercellongg{$1,0$}{$^{s^0}$}}
\put( 0,0){\numbercellongg{$0,1$}{$^{s^0}$}}
\end{picture}
\ \ \ \ \
\begin{picture}(80,55)(-10,-20)
\put(-10, 8){$B$}
\put(30,30){$L$}
\put(30,-20){$s^1$}
\put( 0,0){\numbercellongg{$0,1$}{$^{s^1}$}}
\end{picture}
\ \ \ \ \
\begin{picture}(80,55)(-10,-20)
\put(-10, 8){$B$}
\put(30,30){$L$}
\put(30,-20){$s^2$}
\put( 0,0){\numbercellongg{$1,0$}{$^{s^2}$}}
\end{picture}
}
\centerline{Figure \arabic{figurecounter}: The game in Example~\ref{example:big match}.}
\addtocounter{figurecounter}{1}
\bigskip

Let $\calD^1$ be the partition of $S$ in which each state forms a distinct element:
\[ \calD^1 := \bigl\{ \{s^0\}, \{s^1\}, \{s^2\} \bigr\}. \]
Let $c^1$ be the vector composed of the uniform value of Player~1 in the various states:
\[ c^1(\{s^0\}) := \tfrac{1}{2}, \ \ \ c^1(\{s^1\}) := 0, \ \ \ c^1(\{s^2\}) := 1. \]
By Blackwell and Ferguson (1968) and Theorem~\ref{theorem:5},
\[ \widehat {\underline v}^1_1(s_0;\calD^1,c^1) = \underline v^1_1(s_0) = \tfrac{1}{2}. \]

We will show that the inequality in Theorem~\ref{theorem:maxmin:comparison} can be strict by showing that for every discount factor $\lambda \in [0,1)$
we have
\[ \widehat{\underline v}^1_{\lambda,stat}(s^0;\calD^1,c^1) =
\tfrac{1}{3}. \]
Any stationary strategy $[p(T),(1-p)(B)]$ of Player~1 determines the total $\lambda$-discounted time $\alpha$ that the play spends in $s^0$ before it is absorbed.
This $\lambda$-discounted time is denoted by $\alpha = \alpha_\lambda(p)$ and is given by
\begin{eqnarray}
\nonumber
\alpha_\lambda(p) &=& (1-\lambda) + \lambda(1-\lambda)(1-p) + \lambda^2(1-\lambda)(1-p)^2 + \cdots\\
&=& \frac{1-\lambda}{1-\lambda(1-p)}.
\end{eqnarray}
\label{equ:alpha}
The function $p \mapsto \alpha_\lambda(p)$ given in Eq.~(\ref{equ:alpha}) is a monotone decreasing function from $[0,1]$ to $[1-\lambda,1]$ and onto.
It follows that
\begin{eqnarray}
\nonumber
\widehat{\underline v}^1_{\lambda,stat}(s^0;\calD^1,c^1)
&=&
\max_{x^1 \in \Sigma^1_{\tiny{stat}}} \min_{x^{2} \in \Sigma^2_{\tiny{stat}}} \widehat \gamma_\lambda^1(s_0,x^1,x^{2};\calD^1,c^1)\\
&=& \max_{\alpha \in [1-\lambda,1]} \min_{y \in [0,1]} \left( \alpha \cdot \min\{1-y,\tfrac{1}{2}\} + (1-\alpha)y \right)
\label{equ:591}\\
&=& \tfrac{1}{3},
\nonumber
\end{eqnarray}
and the maximum in Eq.~(\ref{equ:591}) is attained at $\alpha=\tfrac{2}{3}$.
\end{example}

\section{Application: Uniform Equilibrium in Strongly Controllable Games}
\label{section:application}

In this section we show that the modified game can be used to prove the existence of a uniform $\ep$-equilibrium in a certain
class of multiplayer stochastic games.
We start by describing the class of games that we study.
Call a set of states $D$ \emph{closed} if the play cannot leave $D$, whatever the players play.
Call the set $D$ \emph{strongly controllable} if the play can leave $D$ only when the play is in a specific state $s_D \in D$,
and a specific player $i_D$ plays a specific action $\widetilde a^{i_D}$;
in all other states in $D$ the play cannot leave $D$,
and when the play is in state $s_D$ and player $i_D$ plays an action that is not $\widetilde a^{i_D}$, the play cannot leave $D$.

Let $\calD$ be the partition of the set of states according to the vector of uniform min-max values of the players:
two states are in the same element of the partition if all players have the same uniform min-max values in the two states.
Consider the following refinement $\calD^*$ of $\calD$:
two states $s$ and $s'$ are in the same element of $\calD^*$ if
(a) they lie in the same element of $\calD$,
(b) there is a strategy profile that ensures that the play reaches $s'$ without leaving $D$ when the initial state is $s$,
and (c) there is a strategy profile that ensure that the play reaches $s$ without leaving $D$ when the initial state is $s'$.
We call a stochastic game \emph{strongly controllable} if each element of the partition $\calD^*$ is either closed or strongly controllable.

We will prove that every strongly controllable multiplayer stochastic game admits a uniform equilibrium payoff.
Admittedly, this class of games is restricted,
yet the proof that is described below and shows that games in this class admit a uniform $\ep$-equilibrium,
which involves the modified game,
is the first to achieve this task.
An alternative way to prove the existence of a uniform $\ep$-equilibrium in this class of games uses the technique of Solan and Vieille (2002, Section~4).

Strongly controllable games are by no means significant in their own right.
The goal of this section is to show a proof of concept:
the modified game can be used to prove the existence of a uniform equilibrium payoff in some nontrivial class of multiplayer stochastic games.
We hope that in the future more convincing applications to this tool will be found.

\begin{definition}
Let $\ep \geq 0$.
A strategy profile $\sigma$ is a \emph{uniform $\ep$-equilibrium}
if for every initial state $s_0 \in S$,
every player $i \in I$,
every integer $N \in \dN$ sufficiently large,
and every discount factor $\lambda \in [0,1)$ sufficiently close to 1,
\[ \gamma^i_\lambda(s_0,{\sigma'}^i,\sigma^{-i}) \leq \gamma^i_\lambda(s_0,\sigma) + \ep, \ \ \ \forall {\sigma'}^i \in \Sigma^i, \]
and
\[ \gamma^i_N(s_0,{\sigma'}^i,\sigma^{-i}) \leq \gamma^i_N(s_0,\sigma) + \ep, \ \ \ \forall {\sigma'}^i \in \Sigma^i. \]
\end{definition}

The existence of a uniform $\ep$-equilibrium has been verified for various classes of stochastic games
(see, e.g., Mertens and Neyman (1981), Solan (1999), Vieille (2000a,b), Solan and Vieille (2001),
Flesch, Thuijsman, and Vrieze (2007), Simon (2007, 2012), and Flesch, Schoenmakers, and Vrieze (2008, 2009)).
In this section we concentrate on the class of strongly controllable stochastic games,
which we define now formally.

\begin{definition}
\label{def:controllable:set}
A set of states $D \subseteq S$ is \emph{closed} if $q(D \mid s,a) = 1$ for every state $s \in D$ and every action profile $a \in A(s)$.
A set of states $D \subseteq S$ is \emph{strongly controllable} if the following condition holds:
there is a player $i_D \in I$, a state $s_D \in D$,
and an action $\widetilde a^{i_D} \in A^{i_D}(s_D)$ such that for every state $s \in D$ and every action profile $a \in A(s)$,
if $q(D \mid s,a) < 1$
then $s = s_D$ and $a^{i_D} = \widetilde a^{i_D}$.
\end{definition}

Let $\calD$ be the partition of the set of states according to the uniform min-max value vector:
two states $s,s' \in S$ are in the same element of $\calD$ if and only if $\overline v^i_1(s) =\overline v^i_1(s')$ for every player $i \in I$.

For every set $C \subseteq S$
denote by $\nu_C$ the first arrival time to $C$:
\[ \nu_C := \min\{ n \geq 0 \colon s_n \in C\}. \]
By convention, the minimum of an empty set is $+\infty$.

\begin{definition}
Let $D \subseteq S$ be a set of states and let $s,s' \in D$.
We say that \emph{state $s$ leads in $D$ to state $s'$} if there is a strategy profile $\sigma$ such that
under $\sigma$ the play reaches $s'$ before leaving $D$ when the initial state is $s$:
\begin{equation}
\label{equ:leadto}
\prob_{s,\sigma}(\nu_{\{s'\}} < \nu_{S \setminus D}) = 1.
\end{equation}
The states $s$ and $s'$ are \emph{$D$-siblings} if each one leads in $D$ to the other.
\end{definition}

If there is a strategy profile $\sigma$ that satisfies Eq.~(\ref{equ:leadto}),
then there is a pure stationary strategy profile that satisfies this equation,
and, moreover, this strategy profile is independent of $s$, as soon as $s\in D$.
By definition, every state leads to itself in any set $D$ that contains it.
As a result, the $D$-siblings relation is reflexive, commutative, and transitive.
The following result states that if the states $s$ and $s'$ are $D$-siblings,
and if $s''$ is a state that is visited with positive probability under the strategy that leads in $D$ from $s$ to $s'$,
then the states $s'$ and $s''$ are $D$-siblings as well.

\begin{lemma}
\label{lemma:siblings}
Let $D \subseteq S$ be a set of states, let $s,s' \in D$,
and let $\sigma$ be a strategy profile that satisfies Eq.~(\ref{equ:leadto}).
Let $s'' \in D$ be a state that satisfies $\prob_{s,\sigma}(\nu_{\{s''\}} < \nu_{\{s'\}}) > 0$.
If $s$ and $s'$ are $D$-siblings, then $s'$ and $s''$ are $D$-siblings.
\end{lemma}

\begin{proof}
The part of $\sigma$ from its first visit to $s''$ until the play reaches $s'$ is a strategy profile that leads from $s''$ to $s'$
without leaving $D$.

We now prove that there is a strategy profile that ensures that the play reaches $s''$ without leaving $D$ when the initial state is $s'$.
Let $\sigma'$ be a strategy profile that ensures that the play reaches state $s$ without leaving $D$ when the initial state is $s'$.
Let $\sigma''$ be the strategy profile in which the players alternately follow $\sigma'$ until the play reaches $s$
and then they follow $\sigma$ until the play reaches $s'$.
By definition, under $\sigma''$ the play never leaves $D$, that is, $\prob_{s',\sigma''}(\nu_{S \setminus D}=\infty) = 1$.
We claim that under $\sigma''$ the play reaches $s''$ with probability 1, that is, $\prob_{s',\sigma''}(\nu_{\{s''\}} < \infty) = 1$.
Indeed, after each visit to $s$, with probability $\prob_{s,\sigma}(\nu_{\{s''\}} < \nu_{\{s'\}})$ the play reaches state $s''$ before reaching state $s'$,
hence the play eventually reaches state $s''$ a.s.
\end{proof}

\bigskip

Two states $s,s'$ are $\calD$-siblings if they are $D$-siblings for some element $D \in \calD$.
Denote by $\calD^*$ the partition of the set of states according to the $\calD$-siblings equivalence relation.
This partition is a refinement of the partition $\calD$.
Lemma~\ref{lemma:siblings} implies that if the states $s$ and $s'$ lie in the same element of the partition $\calD^*$,
then state $s$ leads to state $s'$ without leaving the element of $\calD^*$ that contains both states $s$ and $s'$.

\begin{definition}
\label{def:strongly}
A stochastic game is \emph{strongly controllable} if all
elements $D \in \calD^*$ are either closed or strongly controllable.
\end{definition}

Closed sets of states resemble repeated games, since the play cannot leave such a set,
can move between the states, and the min-max value is independent of the state.
Consequently, whenever the initial state is in a closed set of states,
a uniform $\ep$-equilibrium exists.
In fact, in this case a folk theorem obtains,
and every feasible and individually rational payoff vector is a uniform equilibrium payoff,
namely a limit as $\ep$ goes to 0 of payoff vectors that correspond to uniform $\ep$-equilibria.
The main result of this section is the following.

\begin{theorem}
\label{theorem:app:1}
Every strongly controllable stochastic game $\Gamma$ admits a uniform $\ep$-equilibrium, for every $\ep > 0$.
\end{theorem}

The rest of the section is devoted to the proof of Theorem~\ref{theorem:app:1},
which uses the approach of Vrieze and Thuijsman (1989) and Solan (1999), and, except of the use of the modified game, is by now standard.
We will therefore provide the main lines of the proof without getting into all the details.

Fix $\ep > 0$ throughout, and for the time being also fix an element $D \in \calD^*$.

\bigskip
\noindent\textbf{Step 1:} Definition of a restricted stochastic game and its modified game.

Consider the stochastic game $\Gamma^D$ that is derived from the stochastic game $\Gamma$ by restricting the set of states to $D$; that is,
\begin{itemize}
\item   The set of players is $I$.
\item   The set of states is $S$; all states $s \not\in D$ are absorbing with absorbing payoff $\overline v_1(s)$.
\item   The set of actions of each player~$i$ in each state $s \in D$ is $A^i(s)$.
\item   Transitions and stage payoff in states in $D$ are as in the original stochastic game.
\end{itemize}

If $D$ is a closed set of states, let $s_D$ be an arbitrary state in $D$;
if $D$ is a strongly controllable set, the state $s_D$ was defined in Definition~\ref{def:controllable:set}.
Define a partition $\calD_D$ over the set of states $S$ by
\[ \calD_D := \{D, \{s\}_{s \not\in D}\}, \]
and for every player $i \in I$ define a vector of cut-offs $(c^i(E))_{E \in \calD_D}$ by
\[ c^i_D(D) := \overline v^i_1(s_D), \ \ \ \ \ c^i_D(\{s\}) := \overline v^i_1(s) \ \ \ \forall s \not\in D. \]
Since the set $D$ contains all nonabsorbing states in the game $\Gamma^D$,
the partition $\calD_D$ satisfies Property P' w.r.t.~every player with $\zeta=1$ (see the paragraph after Definition~\ref{def:property:ptag}).

Consider the modified game $\widehat\Gamma_\lambda^D(s_D;\calD_D,\vec c_D)$ that is based on the restricted stochastic game $\Gamma^D$,
where all players share the same partition $\calD_D$, and the vector of cutoffs is $\vec c := (c^i_D)_{i \in I}$.
For every $\lambda \in [0,1)$ let $x_\lambda$ be a stationary equilibrium in the modified game $\widehat\Gamma_\lambda^D(s_D;\calD_D,\vec c_D)$
such that the function $\lambda \mapsto x_\lambda$ is semi-algebraic (see Theorem~\ref{theorem:eq:semialgebraic}).
In particular the limit $x_1(s) := \lim_{\lambda \to 1} x_\lambda(s)$ exists for every state $s \in S$,
and by Corollary~\ref{corollary:equilibria:minmax} we have
\begin{equation}
\label{equ:lim:payoff}
\lim_{\lambda \to 1} \widehat\gamma^i_\lambda(s_0,x_\lambda;\calD_D,c^i_D) \geq \overline v^i_1(s_D).
\end{equation}

The payoff in the modified game $\widehat\Gamma_\lambda^D(s_D;\calD_D,\vec c_D)$ under the strategy profile $x_\lambda$ is
\begin{eqnarray}
\label{equ:payoff:4}
\widehat\gamma^{D,i}_\lambda(s_D,x_\lambda;\calD_D,\overline v_1^i) &=&
\sum_{E \in \calD} \min\left\{U^i_\lambda(s_D,x_\lambda;E), t_\lambda(s_0,x_\lambda;E)\cdot  \overline c_D^i(E)\right\}\\
&=& \min\left\{U^i_\lambda(s_D,x_\lambda;D), t_\lambda(s_D,x_\lambda;D)\cdot  \overline v_1^i(s_D)\right\}
+ \sum_{s \not\in D} t_\lambda(s_D,x_\lambda;\{s\}) \cdot \overline v_1^i(s).
\nonumber
\end{eqnarray}
By taking the limit as $\lambda$ goes to 1 in Eq.~(\ref{equ:payoff:4}) and using Eq.~(\ref{equ:lim:payoff}) we deduce that
exactly one of the following alternatives hold:
\begin{enumerate}
\item[(A.1)] $\lim_{\lambda \to 1} t_\lambda(s_D,x_\lambda;D) = 1$ and $\lim_{\lambda \to 1} U^i_\lambda(s_D,x_\lambda;D) \geq \overline v^i_1(s_D)$
for every player $i \in I$.
\item[(A.2)] $\lim_{\lambda \to 1} t_\lambda(s_D,x_\lambda;D) < 1$ and
$\sum_{s \not\in D} t_\lambda(s_D,x_\lambda;\{s\}) \cdot \overline v_1^i(s) \geq \overline v^i_1(s_D)$ for every player $i \in I$.
\end{enumerate}
The significance of the modified game is in implying this dichotomy.
We will show that if for $D \in \calD^*$ Condition (A.1) holds,
then when the initial state is in $D$ there is a uniform $\ep$-equilibrium under which the play never leaves $D$,
while if Condition (A.2) holds,
then when the initial state is in $D$ there is a uniform $\ep$-equilibrium under which the play leaves $D$ with probability 1.

\bigskip
\noindent\textbf{Step 2:} Condition (A.1) holds.

The condition $\lim_{\lambda \to 1} t_\lambda(s_D,\sigma;D) = 1$ arises in two cases:
\begin{itemize}
\item   The set $D$ is closed, and in particular the play cannot leave it.
\item   The set $D$ is strongly controllable, and as $\lambda$ goes to 1 the $\lambda$-discounted time at which the play leaves $D$ under $x_\lambda$ goes to 0:
$\lim_{\lambda \to 1}\E_{s_0,x_\lambda}[(1-\lambda)\lambda^{\nu_{S \setminus D}}] = 0$.
\end{itemize}
Both cases are treated in the same way.

Since all states in the set $D$ are $D$-siblings (Lemma~\ref{lemma:siblings}),
for every state $s' \in D$ there is a stationary strategy profile
that ensures that the play reaches state $s'$ without leaving $D$ when the initial state is in $D$.
If $t_\lambda(s_D,\sigma;\{s'\}) > 0$ for every $\lambda \in [0,1)$,
that is, the $\lambda$-discounted time spent in state $s'$ is positive 0,
then there is a stationary strategy profile
that is a perturbation of $x_1$ and ensures that the play reaches $s'$ without leaving $D$ when the initial state is in $D$.
Formally, for every $\delta > 0$ there is a stationary strategy profile $y_{\delta,s'}$ that satisfies
\begin{itemize}
\item   The strategy profile $y_{\delta,s'}$ is a perturbation of $x_1$, that is, $\|y^i_{\delta,s'}(s) - x_1^i(s)\|_\infty < \delta$
for every player $i \in I$ and every state $s \in D$.
\item   Under the strategy profile $y_{\delta,s'}$ the play reaches state $s'$ with probability 1 without leaving $D$,
provided the initial state is in $D$:
\[ \prob_{s,y_{\delta,s'}}(\nu_{\{s'\}} < \nu_{S \setminus D}) = 1, \ \ \ \forall s \in D. \]
\end{itemize}

Since $\lim_{\lambda \to 1} t_\lambda(s_D,x_\lambda;D) = 1$, Eq.~(\ref{equ:106}) implies that
\[ \lim_{\lambda \to 1}U^i_\lambda(s_D,x_\lambda;D) = \lim_{\lambda \to 1}\gamma^i_\lambda(s_D,x_\lambda), \]
hence Condition (A.1) implies that
\[ \lim_{\lambda \to 1}\gamma^i_\lambda(s_D,x_\lambda) \geq \overline v_1^i(s_D), \ \ \ \forall i\in I. \]
The quantity $\lim_{\lambda \to 1}\gamma_\lambda(s_D,x_\lambda)$ is a convex combination of the payoff under the stationary strategy profile $x_1$
in its irreducible sets.
Recall that a nonempty set $C \subseteq S$ is \emph{irreducible under $x_1$} if
\begin{itemize}
\item[(I.1)] The set $C$ is closed under $x_1$, that is, $q(C \mid s,x_1) = 1$ for every state $s \in C$.
\item[(I.2)] The set $C$ is a minimal set (w.r.t.~set inclusion) that satisfies property (I.1).
\end{itemize}
Denote by $\calI(D;x_1)$ the collection of all irreducible sets under $x_1$ that are contained in $D$.
It is well known that for every irreducible set $C$ under $x_1$
the limit $\lim_{\lambda \to 1} \gamma_\lambda(s_0,x_1)$ is independent of the initial state $s_0$,
as long as $s_0$ is in $C$.
We denote this limit by $\gamma_1(C,x_1)$.
For every irreducible set $C$ under $x_1$ denote the $\lambda$-discounted time that the play spends in the set $C$ under strategy profile $x_\lambda$
when the initial state is $s_D$ by
\[ \beta(C) := \lim_{\lambda \to 1} t_\lambda(s_D,x_\lambda;C). \]
Then
\[ \lim_{\lambda \to 1}\gamma^i_\lambda(s_D,x_\lambda) = \sum_{C \in \calI(D;x_1)} \beta(C) \cdot \gamma_1(C,x_1). \]
Denote $\supp(\beta) = \{C_1,C_2,\cdots,C_L\}$
and for each $l \in \{1,2,\cdots,L\}$ choose a state $s_l \in C_l$.

Consider the following strategy profile $\sigma_K$ that depends on a positive real number $K$:
\begin{enumerate}
\item   Set $l=1$.
\item   The players play the stationary strategy profile $y_{\ep,s_l}$ until the play reaches a state in $C_l$.
\item   The players play the stationary strategy profile $x_1$ for $\lceil \beta(C_l) \cdot K\rceil$ stages.
\item   Increase $l$ by 1 modulo $L$ and go to Step 2.
\end{enumerate}

The reader can verify that as $K$ increases the payoff under $\sigma_K$ converges to $\sum_{C \in \calI(D;x_1)} \beta(C) \cdot \gamma_1(C,x_1)$, that is,
\begin{equation*}
\lim_{K \to \infty} \lim_{\lambda \to 1} \gamma^i_\lambda(s_D,\sigma_K) = \sum_{C \in \calI(D;x_1)} \beta(C) \cdot \gamma_1(C,x_1),\\
\end{equation*}
and
\begin{equation*}
\lim_{K \to \infty} \lim_{N \to \infty} \gamma^i_N(s_D,\sigma_K) = \sum_{C \in \calI(D;x_1)} \beta(C) \cdot \gamma_1(C,x_1).
\end{equation*}
It is standard to verify that by adding statistical tests to $\sigma_K$,
which check whether the players indeed follow it,
the strategy profile $\sigma_K$ can be turned into a uniform $2\ep$-equilibrium, provided the initial state is in $D$.

Choose $K_0 \in \dN$ sufficiently large such that
\[
\left|\lim_{\lambda \to 1} \gamma^i_\lambda(s_D,\sigma_{K_0}) - \sum_{C \in \calI(D;x_1)} \beta(C) \cdot \gamma_1(C,x_1)\right| \leq \ep,
\]
and
\begin{equation*}
\left| \lim_{N \to \infty} \gamma^i_N(s_D,\sigma_{K_0}) - \sum_{C \in \calI(D;x_1)} \beta(C) \cdot \gamma_1(C,x_1)\right| \leq \ep.
\end{equation*}

\bigskip
\noindent\textbf{Step 3:} Condition (A.2) holds.

The condition $\lim_{\lambda \to 1} t_\lambda(s_D,x_\lambda;D) < 1$ implies that the play can leave $D$,
hence in particular the set $D$ is not closed, and therefore it is strongly controllable.
By Condition (A.2),
\[ \sum_{s \not\in D} \lim_{\lambda \to 1} t_\lambda(s_D,x_\lambda;\{s\}) \cdot \overline v_1^i(s) \geq \overline v_1^i(s_D), \ \ \ \forall i \in I, \]
and therefore
\[ \lim_{\lambda \to 1}
 \left(\frac{\sum_{s' \not\in D} q(s' \mid s_D,\widetilde a^{i_D},x^{-i_D}_\lambda) \cdot \overline v_1^i(s')}
 {q(S \setminus D \mid s_D,\widetilde a^{i_D},x^{-i_D}_\lambda)}\right) \geq \overline v_1^i(s_D), \ \ \ \forall i \in I. \]
A strategy profile that ensures that the play leaves the set $D$ and the expected continuation uniform min-max value is high suggests itself.
Fix $\lambda_0 \in [0,1)$ that satisfies
\begin{equation*}
\frac{\sum_{s' \not\in D} q(s' \mid s_D,\widetilde a^{i_D},x^{-i_D}_{\lambda_0}) \cdot \overline v_1^i(s')}
 {q(S \setminus D \mid s_D,\widetilde a^{i_D},x^{-i_D}_{\lambda_0})} \geq \overline v_1^i(s_D)-\delta,
 \ \ \ \forall i \in I,
 \end{equation*}
 where $\delta > 0$ will be determined below.
\begin{itemize}
\item   The players play a pure stationary strategy that guarantees that the play reaches the state $s_D$.
\item   At state $s_D$ players $I \setminus \{i_D\}$ play $x^{-i_D}_{\lambda_0}$, while player $i_D$ plays $(1-\eta) \widehat a^{i_D} + \eta \widetilde a^{i_D}$,
where $\widehat a^{i_D}$ is some action in $A^{i_D}(s_D) \setminus \{\widetilde a^{i_D}\}$,
and $\eta>0$ is sufficiently small.
\item   Deviations in states in $D \setminus \{s_D\}$ are observed immediately and punished at the deviator's uniform min-max level.
\item   Deviations in $s_D$ of players in $I \setminus \{i_D\}$ are observed statistically, provided $\eta$ is sufficiently small.
\item   If the play has not left $D$ after $\tfrac{1}{\eta^2}$ visits to state $s_D$
and no deviation has been detected,
all players start punishing player~$i_D$ at his uniform min-max level.
The constant $\eta$ is chosen sufficiently small so that if no player deviates,
the probability that the play does not leave $D$ after $\tfrac{1}{\eta^2}$ visits to state $s_D$ is smaller than $\delta$.
\end{itemize}

\bigskip
\noindent\textbf{Step 4:} Combining the strategies.

Denote by $\sigma^*$ the strategy profile that plays as follows:
\begin{itemize}
\item   Whenever the play enters an element $D \in \calD^*$ that satisfies $\lim_{\lambda \to 1} t_\lambda(s_D,x_\lambda;D) = 1$,
the strategy profile $\sigma^*$ follows the strategy profile $\sigma_{K_0}$ defined in Step~2.
\item   Whenever the play enters an element $D \in \calD^*$ that satisfies $\lim_{\lambda \to 1} t_\lambda(s_D,x_\lambda;D) < 1$,
the strategy profile $\sigma^*$ follows the strategy profile defined in Step~3 until the play leaves the set $D$.
\end{itemize}

%For every $k \geq 0$ denote by $z^i_k = \overline v_1^i(s^{\nu^{\calD^*}_k})$ the uniform min-max value in the $k$'th element of $\calD^*$ that the play visits.
The strategy profiles defined in Step~2 ensures that, when the initial state is in an element $D \in \calD^*$ that satisfies Condition (A.1),
the discounted payoff and the average payoff
of each player~$i$ is at least $\overline v_1^i(s_D)-2\delta$, provided the discount factor is sufficiently close to 1 and the horizon is sufficiently long.
The strategy profile defined in Step~3 ensures that, when the initial state is in an element $D \in \calD^*$ that satisfies Condition (A.2),
the play leaves $D$
and the expected continuation uniform min-max value of each player~$i$ is at least $\overline v_1^i(s_D)-\delta$.
Note that for every player $i \in I$, the strategy $\sigma^{*i}$ is $\calD_D$-Markovian.

Recall that $k(\calD^*;n)$ is the number of times in which the play switched elements of the partition $\calD^*$ up to stage $n$.
For each player $i \in I$ define a stochastic process $(W_n)_{n \geq 0}$ with values in $\dR^I$ by $W^i_n := \overline v^i_1(s_n) + \delta \cdot k(\calD^*;n)$
for each $i \in I$ and every $n \geq 0$.
By the construction of the strategy profile $\sigma^*$,
the process $(W^i_n)_{n \geq 0}$ is a bounded submartingale under the strategy profile $\sigma^*$, for every player $i \in I$.

We argue that the expected number of times in which the play switches between elements of $\calD^*$ is bounded by a constant
that is independent of $\delta$.
Indeed, the proof of Theorem~\ref{theorem:partition2}
implies that the number of times in which the play switches between elements of $\calD$ is bounded by $C(\rho)$.
Moreover, if $s,s'$ are two states in the same elements of $\calD$ but in different elements of $\calD^*$,
then one of them does not lead to the other.
It follows that the expected number of times in which the process changes an element of the partition $\calD^*$
before leaving the current element of the partition $\calD$ is uniformly bounded by a constant $C'$
that depends only on the transition function $q$. The claim follows.

As in the proof of Theorem~\ref{theorem:partition2} we deduce that there is an event $E$ that satisfies the following conditions:
\begin{itemize}
\item   $\prob_{s_0,\sigma^*}(E)  \geq 1-\ep$ for every initial state $s_0$.
\item   On the event $E$ the number of times the process $(W_k)_{k \geq 0}$ changes its values is uniformly bounded, say by $C''$,
which depends on $\rho$ and on the transition function $q$.
\end{itemize}
We choose $\delta$ sufficiently small such that $C'' \delta \leq \ep$.

We deduce that $\prob_{s_0,\sigma^*}$-a.s.~the play reaches a closed element in $\calD^*$,
and therefore the expected long-run average payoff of each player~$i$ under $\sigma^*$ is at least
$\overline v^i_1(s_0) - 3\ep$.
Moreover, for every history $h_n \in H$ the expected long-run average payoff of each player~$i$ under $\sigma^*$,
conditioned that history $h_n$ occurred, is at least
$\overline v^i_1(s_n) - 3\ep$.
This in turn implies that the strategy profile $\sigma^*$ is uniform $4\ep$-equilibrium.

\section{Discussion and Open Problems}

In this paper we defined for every stochastic game an auxiliary game, which we termed the modified game,
studied some of its properties, and showed that equilibria in this auxiliary game can be used to study uniform equilibrium in the original stochastic game.
This approach is similar to the vanishing discount factor approach of Vrieze and Thuijsman (1989),
who studied uniform equilibrium in two-player absorbing games by analyzing an auxiliary game,
which, in their case, was the discounted game.
Solan (1999) took this approach one step further by studying a more involved version of the payoff function in absorbing games,
by altering the payoff function in the nonabsorbing state.
In the present paper we extended Solan's (1999) approach to general stochastic games.
Vieille (2000b), Solan (2000), and Solan and Vieille (2002) studied another type of auxiliary game when the underlying game is a recursive game
by restricting the players to completely mixed strategies.

Unlike Vrieze and Thuijsman (1989), Solan (1999), Vieille (2000b), Solan (2000), and Solan and Vieille (2002), the auxiliary game presented in this paper
is valid for every stochastic game, and not only for absorbing games or for recursive games.
It is probable that our modified game is not the only auxiliary game that can be useful to studying general stochastic games,
and that in the future other types of auxiliary games will be proposed and studied.

Example~\ref{example:1} shows that there need not be a stationary strategy profile $x$ that satisfies Eq.~(\ref{equ:500})
for every player $i \in I$ and every initial state $s_0 \in S$.
We do not know whether for every $\ep > 0$ there exists a stationary strategy profile $x$ that satisfies
\[ \widehat\gamma_\lambda^i(s_0,x;\calD^i,c^i) \geq \overline v^i_1(s_0) - \ep, \ \ \ \forall i\in I, s_0 \in S. \]
The existence of such a stationary strategy profile may have significant implications on the study of acceptability in stochastic games
(Solan, 2016) and on the study of uniform correlated $\ep$-equilibrium in stochastic games (Solan and Vieille, 2002; Solan and Vohra, 2002).

As mentioned before, we do not know whether the functions
$\lambda \mapsto \underline{\widehat v}^i_{\lambda}(s_0;\calD^i,c^i)$
and
$\lambda \mapsto \overline{\widehat v}^i_{\lambda}(s_0;\calD^i,c^i)$
are semi-algebraic for every initial state $s_0 \in S$,
every player $i \in I$,
every partition $\calD^i$ of the set of states,
and every vector $c^i \in \dR^{\calD^i}$.
It is interesting to know whether this is indeed the case,
and if not, whether the limits of these functions as the discount factor goes to 1 nevertheless exist.

\end{document}